\documentclass[12pt]{article} 
\usepackage[utf8]{inputenc}
\usepackage[T1]{fontenc}
\usepackage[english]{babel}

\usepackage{mathrsfs}

\usepackage[stretch=10]{microtype}
\usepackage{amsmath}
\usepackage{amsthm}
\usepackage{amssymb} 
\usepackage{mathtools}
\usepackage{enumitem}

\usepackage[textwidth=6in]{geometry}

\usepackage{hyperref}

\allowdisplaybreaks

\newcommand{\hlint}[2]{\left({#1}, {#2}\right]}

\newcommand{\CC}{\mathbb{C}}
\newcommand{\RR}{\mathbb{R}}

\newcommand{\cB}{\mathcal{B}}

\newcommand{\cD}{\mathcal{D}}
\newcommand{\cH}{\mathcal{H}}

\newcommand{\cS}{\mathcal{S}}

\newcommand{\D}{\mathrm{d}}

\renewcommand{\i}{{\mathrm{i}}}

\newtheorem{thm}{Theorem}
\newtheorem{prop}[thm]{Proposition}

\newtheorem{lemma}[thm]{Lemma}
\newtheorem{defn}[thm]{Definition}
\newtheorem{cor}[thm]{Corollary}
\numberwithin{thm}{section} 

\theoremstyle{remark}
\newtheorem{rmk}[thm]{Remark}

\newcommand{\astfootnote}[1]{%
\let\oldthefootnote=\thefootnote%
\setcounter{footnote}{0}%
\renewcommand{\thefootnote}{\fnsymbol{footnote}}%
\footnote{#1}%
\let\thefootnote=\oldthefootnote%
}
\newcommand{\secondastfootnote}[1]{%
\let\oldthefootnote=\thefootnote%
\setcounter{footnote}{6}%
\renewcommand{\thefootnote}{\fnsymbol{footnote}}%
\footnote{#1}%
\let\thefootnote=\oldthefootnote%
}

\usepackage{xcolor}

\newcommand{\p}{\partial}

\DeclareMathOperator{\diam}{diam}
\DeclareMathOperator{\dist}{dist}
\DeclareMathOperator{\vol}{vol}

\DeclareMathOperator{\BMO}{BMO}
\DeclareMathOperator{\Tr}{Tr}
\DeclareMathOperator{\re}{\mbox{Re} \, } 
\DeclareMathOperator{\im}{\mbox{Im} \, } 

\title{The homogeneous and inhomogeneous Dirichlet problem \\ \vspace{12pt} \large Dedicated to Marius Mitrea on his 60\textsuperscript{th} birthday, with friendship and appreciation}
\author{David S. Jerison\astfootnote{Supported in part by Simons Foundation Collaboration Grant 601948 DJ} \and Carlos E. Kenig\secondastfootnote{Supported in part by NSF grant DMS-2153794}}

\usepackage[backend=biber, style=alphabetic,maxnames=50,maxalphanames=50,sorting=nyt,safeinputenc]{biblatex}
\DeclareNameAlias{author}{last-first}

\begin{document}
  \maketitle
\begin{abstract} 
We revisit the homogeneous and inhomogeneous Dirichlet problem for the Laplacian on Lipschitz domains. This is motivated by the recent postings by Amrouche and Moussaoui which purport to contradict known area integral estimates of Dahlberg and
known Sobolev space estimates of Jerison and Kenig. We explain the fatal gap in the reasoning in these postings and give a self-contained proof of a special case of the results of Dahlberg.  We then show that this is sufficient to disprove
the central conclusions of these postings. We also provide two further proofs of the results of Dahlberg, adapted from 
work of Kenig and of Dahlberg-Kenig-Pipher-Verchota.   Other proofs are in the original paper by Dahlberg
and in work by Fabes-Mendez-Mitrea.
\end{abstract}

  \bigskip
  
  The classical theorem concerning the inhomogeneous Dirichlet problem (or Poisson problem) says that for every bounded, open set $\Omega \subset \RR^n$ and every $f \in H^{-1}(\Omega)$, there is a unique $u \in H_0^1(\Omega)$   that solves
  \begin{align}
    \Delta u \
    &= \ f.
    \label{eq:1}
  \end{align}
(The spaces $H^{-1}$, $H_0^1$, $H^2$, $H^s$ are the standard $L^2$-based Sobolev spaces.)     Furthermore, $u$ satisfies
  \begin{align}
    \|\nabla u\|_{L^2(\Omega)} \
    &= \ \|f\|_{H^{-1}(\Omega)}.
    \label{eq:2}
  \end{align}
  In other words, $\Delta^{-1}$ is an isometry from $H^{-1}(\Omega)$ to $H_0^1(\Omega)$, provided $H_0^1(\Omega)$ is equipped with the norm $\|u\|_{H_0^1(\Omega)} = \|\nabla u\|_{L^2(\Omega)}$ and $H^{-1}(\Omega)$ with the dual norm.  

  We can only expect higher regularity of $u$ under further assumptions on $f$ and the domain $\Omega$.
  For a bounded Lipschitz domain $\Omega \subset \RR^n$, one can choose the norm $\| \cdot \|_{\ast}$ on the space of $u \in H^2(\Omega) \cap H_0^1(\Omega)$ to be
  \begin{align*}
    \|u\|_{\ast}^2 \
    &= \ \int\limits_{\Omega} \sum_{i, j=1}^{n} \left(\p_i \p_j u \right)^2 \, \D x \quad
\mbox{with} \ \ x = (x_1,\, x_2, \, \dots \, , \, x_n) \ \ \mbox{and} \ \ \p_j = \p/\p x_j. 
  \end{align*}
  In 1963, J. Kadlec proved the optimal second-order regularity for the inhomogeneous Dirichlet problem in convex domains.
  A special case of this theorem is as follows \cite{K64}.
  \begin{thm}
    \label{thm:0.1}
    If $\Omega$ is a bounded, open, convex subset of $\RR^n$, then for every $f \in L^2(\Omega)$ the unique function $u \in H_0^1(\Omega)$ such that $\Delta u = f$ satisfies
    \begin{align*}
      \|u\|_{\ast} \,
      &\leq \, \|f\|_{L^2(\Omega)}.
    \end{align*}
  \end{thm}

  In other words, the inverse mapping $\Delta^{-1} \colon L^2(\Omega) \to H^2(\Omega) \cap H_0^1(\Omega)$ is well-defined and a contraction under the $\| \cdot \|_{\ast}$ norm.
  (For optimal regularity up to third order, see \cite{FJ94}).

  In essence, Kadlec's proof follows from an integration by parts formula in the plane.
  Suppose that $\Omega$ is a smooth, bounded, planar domain, $u \in C^2(\overline{\Omega})$ and $u = 0$ on $\partial \Omega$.
  Then
  \begin{align}
    \int\limits_{\Omega} \left[(\p_{x x} u) (\p_{y y} u) - (\p_{xy} u)^2 \right] \, \D x \, \D y \
    &= \ \int\limits_{\partial \Omega} \kappa \left(\frac{\p u}{\p \nu} \right)^2 \, \D \sigma
    \label{eq:3}
  \end{align}
  with ${\p}/{\p \nu}$ the normal derivative, $\kappa$ the curvature, and $\D \sigma$ the arc length measure of $\partial \Omega$.
  In particular, if $\Omega$ is convex, then $\kappa\ge 0$, and 
  \begin{align*}
    \int\limits_{\Omega} \left(\p_{xy} u\right)^2 \, \D x \, \D y \
    &\leq \ \int\limits_{\Omega} \left(\p_{xx} u \right) \left(\p_{yy} u \right) \, \D x \, \D y.
  \end{align*}
  It follows from Fubini's theorem and by summing over all pairs $(i, j)$ that for smooth, bounded, convex domains $\Omega \subset \RR^n$ and functions $u \in C^2(\overline{\Omega})$ satisfying $u = 0$ on $\partial \Omega$, we have
  \begin{align*}
    \int\limits_{\Omega} \sum_{i, j=1}^{n} \left(\p_i \p_j u \right)^2 \, \D x \
    &\leq \ \int\limits_{\Omega} \left(\Delta u\right)^2 \, \D x, \quad x = (x_1, \ldots, x_n).
  \end{align*}
  From this inequality, Kadlec deduced Theorem \ref{thm:0.1} (see \cite[p.147]{G85}).

  In the case of polygons $\Omega$ in the plane, Grisvard observed in \cite{G72} that 
  \begin{align}
    \int\limits_{\Omega} \left[\left(\p_{xx} u \right) \left(\p_{yy} u \right) - \left(\p_{xy} u \right)^2 \right] \, \D x \, \D y \
    &= \ 0
    \label{eq:4}
  \end{align}
  for all $u \in C^2(\overline{\Omega})$ such that $u \equiv 0$ on $\partial \Omega$.
  An intuitive explanation for why the left-hand side of \eqref{eq:4} is zero for polygons is that $\kappa = 0$ on the sides of the polygon, and $\nabla u = 0$ at the vertices.
  The proof in \cite{G72} involves an additional integration by parts on each side and is valid for many more boundary value problems than the Dirichlet problem\textemdash indeed, any oblique derivative condition on each side.

  It follows easily from Fubini's theorem that the analogous statement holds true for polyhedra in all dimensions.
  Moreover, the identity remains valid for functions in the closure in $H^2(\Omega)$ of this class of functions $u$.
  Thus, Grisvard obtains the following a priori identity.

  \begin{thm}
    \label{thm:0.2}
    If $\Omega$ is a bounded polyhedron in $\RR^n$ and $u \in H^2(\Omega) \cap H_0^1(\Omega)$, then 
    \begin{align*}
      \|u\|_{\ast}^2 \
      &= \ \sum_{i,j = 1}^{n} \|\p_i \p_j u\|^2 \
      = \ \|\Delta u\|^2,
    \end{align*}
    where $\| \cdot \|$ is the $L^2(\Omega)$ norm.
  \end{thm}    

  The difference between Theorem \ref{thm:0.2} and Theorem \ref{thm:0.1} is that for nonconvex polyhedra, the assumptions $\Delta u \in L^2(\Omega)$ and $u \in H_0^1(\Omega)$ do not imply that $u \in H^2(\Omega)$.
  The inequality $\|u\|_{\ast} \leq \|\Delta u\|$ and standard functional analysis imply that the mapping $\Delta \colon H^2(\Omega) \cap H_0^1(\Omega) \to L^2(\Omega)$ is injective and has closed range $R$ in $L^2$.
  When $\Omega$ is convex, the range of $R$ is all of $L^2(\Omega)$.
  If $\Omega$ is a convex polyhedron, then the mappings $\Delta$ and $\Delta^{-1}$ as in Theorem \ref{thm:0.1} are isometries between the $\|\cdot\|_{\ast}$ and the $L^2$-norm.
  But when $\Omega$ is a non-convex polyhedron, the range $R$ is never all of $L^2(\Omega)$.
  In \cite{G72}, Grisvard identifies $R$ for polygons as a space of codimension equal to the number of in-pointing (non-convex) vertices.
  Put another way, Grisvard's identity for polyhedra is only valid if $\|u\|_{\ast} < \infty$, the earmark of an a priori equality or inequality.

  Despite the fact that Grisvard's estimate only applies to solutions to $\Delta u = f$ for $f \in R$, it is tempting to interpolate between his estimate and the classical one, namely, for every bounded open set $\Omega \in \RR^n$ and every $f \in H^{-1}(\Omega)$, there is a unique $u \in H_0^1(\Omega)$ that solves $\Delta u = f$ and
  \begin{align*}
    \|\nabla u \|_{L^2(\Omega)} \
    &= \ \|f\|_{H^{-1}(\Omega)}.
  \end{align*}
  This idea is treated in the recent postings by C. Amrouche and M. Moussaoui \cite{AM22}, \cite{AM25}, \cite{AM26}.
  The articles purport to derive estimates that contradict known estimates for Lipschitz domains in the Sobolev function space $H^{3/2}(\Omega)$ in \cite{JK95}
  as well as estimates of Dahlberg in \cite{D80} saying that the area integral of a harmonic function in a Lipschitz domain controls its $L^2(\partial \Omega)$ norm.
  This calls into question not only the two articles in question but hundreds of papers that cite them. 
  
 The purpose of this paper is, first of all, to explain the fatal gap in the 
 reasoning of Amrouche and Moussaoui, and, secondly, to give a self-contained, direct 
 proof of a special case of  Dahlberg's theorem that will be sufficient to show that  the central conclusions of C. Amrouche and M. Moussaoui in \cite{AM22}, \cite{AM25}, \cite{AM26} are false for Lipschitz domains, although valid for polygonal domains.  We will also describe
 several other proofs of Dahlberg's theorem, including the general case,
 and show how the ideas involved relate to many 
 fundamental theorems in harmonic analysis from the 1960s to the 1990s.
  For a detailed overview of these issues and an analysis in greater depth
  of harmonic functions on non-smooth domains, see the monumental five-volume treatise \cite{MMM22-23}.

  We begin with two assertions of \cite{AM25} that are correct as stated, but which are qualitative in a way that does not permit generalization from polygons to Lipschitz domains.
  Recall first that the space $H_{00}^{1/2}(\Omega)$ is the space of functions whose extension by zero in $\RR^n \setminus \Omega$ belongs to $H^{1/2}(\RR^n)$.
  The dual space $H_{00}^{1/2}(\Omega)'$ signifies $-1/2$ derivatives, and it is the natural target space for the functions $f$ in equation \eqref{eq:1} if we want to consider left-hand sides $u \in H^{3/2}(\Omega) \cap H_0^1(\Omega)$.
  The first assertion, Theorem 5.4 (ii) \cite[page 41]{AM25} is that for polygons $\Omega$,
  \[
\Delta \colon H^{3/2}(\Omega) \cap H_0^1(\Omega) \to H_{00}^{1/2}(\Omega)'
  \]
  is an isomorphism.  
  To state the second assertion, recall that $H^{-1}(\Omega) = \Delta (H_0^1(\Omega))$, and we defined $R = \Delta(H^2(\Omega) \cap H_0^1(\Omega))$.  
  Corollary 5.5 \cite[page 42]{AM25} says that the complex interpolation space 
  \begin{equation} \label{eq:M-space}
    M(\Omega) : = [R, H^{-1}(\Omega)]_{1/2} \ \ \mbox{is isomorphic to} \ \ 
    H_{00}^{1/2}(\Omega)'. 
\end{equation}
  From that point onward in the paper \cite{AM25}, the norms of $M(\Omega)$ and $H_{00}^{1/2}(\Omega)'$ are considered equivalent with constants that don't depend on anything.
  But the proofs of Theorem 5.4 (ii) and Corollary 5.5 are qualitative, proving an isomorphism exists for each polygon without saying how the isomorphism constants depend on the polygon.
 The examples in \cite{AM25} \cite{JK95} in which one considers Lipschitz domains and $C^1$ domains are based on limits of families of polygons for which the number of sides tends to infinity.
 
  Let us now distinguish some true statements in the work of Amrouche and Moussaoui.
  To clarify the comparisons with \cite{G72} and \cite{AM25}, we will specify which versions of various Sobolev norms we are using and their dependence on the polygon $\Omega$.
  For $H_0^1(\Omega)$, we use the Dirichlet norm
  \begin{align*}
    \|g\|_{H_0^1(\Omega)}^2 \
    &= \ \int\limits_{\Omega} |\nabla g|^2 \, \D x.
  \end{align*}
  We use the dual of this norm for the dual space $H^{-1}(\Omega) = H_0^1(\Omega)'$.
  As already mentioned,  for any bounded open subset of $\RR^n$, the Laplacian is an isometry from $H_0^1(\Omega)$ to $H^{-1}(\Omega)$ provided we use these norms.

  We start out by recalling a standard result.
  \begin{lemma}
    \label{lemma:0.3}
    If   $\Omega$ is a bounded Lipschitz domain in $\RR^n$, then the following inequalities hold with  $C(\Omega) = \diam(\Omega)^2/\pi^2$.
    \begin{enumerate}[label=(\alph*)]
      \item For all $f \in H_0^1(\Omega)$,
        \begin{align*}
          \int\limits_{\Omega} |f|^2 \, \D x \
          &\leq \ C(\Omega) \int\limits_{\Omega} |\nabla f|^2 \, \D x.
        \end{align*}
      \item For all $u \in H^2(\Omega) \cap H_0^1(\Omega)$,
        \begin{align*}
          \int\limits_{\Omega} |\nabla u|^2 \, \D x \
          &\leq \ n\,  C(\Omega) \int\limits_{\Omega} |\nabla^2 u|^2 \, \D x.
        \end{align*}
    \end{enumerate}
  \end{lemma}
\begin{proof}
  Part (a) is a direct corollary of the fact that if $f \in C^2[0, \pi]$, and if $f(0) = f(\pi) = 0$, then $\int_{0}^{\pi} f^2 \, \D x \leq \int_{0}^{\pi} |f'|^2 \, \D x$, which can be proved by expressing $f$ as a sine series.
  To prove part (b), observe that the $H^{-1}(\Omega)$ norm is given by
    \begin{align*}
    \|g\|_{H^{-1}(\Omega)} \
    &= \ \sup\left\{g(f) : f \in H_0^1(\Omega), \int_{\Omega} |\nabla f|^2 \, \D x \leq 1 \right\}.
  \end{align*}
  Therefore, if $g \in L^2(\Omega)$ and $f \in H_0^1(\Omega)$, then part (a) implies
  \begin{align*}
    g(f) \,
    &= \, \int\limits_{\Omega} gf \, \D x \
    \leq \, \|g\|_{L^2(\Omega)} \|f\|_{L^2(\Omega)} \,
    \leq \,  C(\Omega)^{1/2} \|g\|_{L^2(\Omega)} \left(\int_{\Omega} |\nabla f|^2 \, \D x \right)^{1/2}.
      \end{align*}
  Hence,  $    \|g\|_{H^{-1}(\Omega)} \ \leq \ C(\Omega)^{1/2} \|g\|_{L^2(\Omega)}$.     Furthermore, since the solution operator of $\Delta^{-1}$ is an isometry from $H^{-1}(\Omega)$ to $H_0^1(\Omega)$, we have, 
    \begin{align*}
    \int\limits_{\Omega} |\nabla u|^2 \, \D x \
    &= \ \|\Delta u\|^2_{H^{-1}(\Omega)} \
    \leq \ C(\Omega) \|\Delta u \|_{L^2(\Omega)}^2 \ 
\leq    \,  n \, C(\Omega) \int\limits_{\Omega} |\nabla^2 u|^2 \, \D x,
  \end{align*}
 for all $u\in H^2(\Omega)\cap H^1_0(\Omega)$.
\end{proof}
  The constants in Lemma \ref{lemma:0.3} account for the extra factors depending on the diameter in the discussion of \cite{AM25}.
  If we use the norm on $H^2(\Omega) \cap H_0^1(\Omega)$ that includes a lower order term, Theorem \ref{thm:0.2} of Grisvard and its restatement in \cite[Theorem 5.3]{AM25} say that, for all $u \in H^2(\Omega) \cap H_0^1(\Omega)$,
  \begin{align*}
    \|u\|_{H^2(\Omega)} \
    &\leq \ C \|\Delta u\|_{L^2(\Omega)},
  \end{align*}
  with $C$ depending only on the diameter of the polygon $\Omega$.

 Since the Laplacian is an isometry for the endpoint norms in \eqref{eq:2} and Theorem \ref{thm:0.2}, complex interpolation implies that for polygons $\Omega$, and for all $u \in H_0^1(\Omega)$ such that $\Delta u \in M(\Omega)$ (the interpolation space defined in \eqref{eq:M-space}),   we have $u \in H^{3/2}(\Omega)$. Indeed, we have an isometry 
  \begin{align*}
    \|u\|_{H^{3/2}(\Omega)} \
    &= \ \|\Delta u\|_{M(\Omega)},
  \end{align*}
  provided we choose for $H^{3/2}(\Omega)$ the $1/2$ complex interpolation norm between the $\| \cdot \|_{\ast}$ norm on $H^2(\Omega) \cap H_0^1(\Omega)$ and the Dirichlet norm on $H_0^1(\Omega)$.
  In the work of \cite{G72} and \cite{AM25}, using lower order terms, one still has the inequality in the interesting direction,
  \begin{align*}
    \|u\|_{H^{3/2}(\Omega)} \
    &\leq \ C \|\Delta u\|_{M(\Omega)},
  \end{align*}
  for a constant $C$ that depends only on the diameter of the polygon $\Omega$.
  This statement, Proposition 5.6 in \cite{AM25}, is correct.
  But in all the subsequent discussion, the two spaces $M(\Omega)$ and $H_{00}^{1/2}(\Omega)'$ are considered to have equivalent norms, without acknowledging that the constants that control the equivalence may vary with the polygon.

 So far our discussion pinpoints a gap in the argument in \cite{AM25}, but does not yet demonstrate that the conclusions are false.
  But the further deductions of \cite{AM25} contradicting the work of Dahlberg \cite{D80} do show that all the consequences in \cite{AM25} concerning Lipschitz domains are incorrect, since Dahlberg's result is correct. 
  
  Our main remaining task is to 
give a complete, self-contained proof of Dahlberg's estimate in \cite{D80} for domains in the plane above the graph of a Lipschitz function $\varphi$ whose Lipschitz norm is small enough.  As we shall see in Theorem 3.1, this special case suffices to show that the approach using Grisvard's estimates \cite{G72} cannot yield the false borderline Sobolev estimates\footnote{The counterexample to the borderline Sobolev estimates is already stated and proved in \cite{JK95}, but in the more general context of $C^1$ domains and in the midst of many other results.}  for the inhomogeneous Dirichlet problem claimed in \cite{AM22}, \cite{AM25}, \cite{AM26}.   
Our proof of the small-constant planar case of Dahlberg's theorem depends on well known properties of conformal mappings and classical estimates from 1970 and earlier.  In particular it does not depend on the ``good lambda'' inequalities used in \cite{D80}.

  We will also provide two other proofs of Dahlberg's estimates.
  One of them is still for domains in the plane above the graph of a Lipschitz function $\varphi$, but with arbitrary Lipschitz norm.
  This proof, in the Appendix, is from \cite{Ke80}.  It relies on the theory of weighted norm inequalities on Euclidean space, a topic at the forefront of harmonic analysis in the 1970's.
  Specifically, it relies on weighted norm inequalities for square functions, proved by R. Gundy and R. Wheeden \cite{GW74}.
  Besides the work of Gundy-Wheeden, this proof still depends on the properties of the conformal mapping, including an  estimate that originates in the breakthrough work of A. P. Calder\'on \cite{C77} on the boundedness of the Cauchy integral operator.
  The Gundy-Wheeden work relies, in turn, on the technique of good lambda inequalities, imported from probability to harmonic analysis in pioneering works of Burkholder, Gundy, and others.
  The third proof we provide works for any bounded Lipschitz domain in $\RR^n$, $n \geq 2$.
  It does not use conformal mappings or good lambda inequalities.
  This proof depends instead on  an ``adapted distance function'', first constructed 
  in a complicated way by  Dahlberg (unpublished).
  Here, we follow \cite{DKPV97}, which gives a simpler construction of
  the adapted distance function due to   Kenig-Stein.
    In \cite{DKPV97}, a very general version of Dahlberg's square function estimates is obtained in the setting of higher order constant coefficient elliptic systems.
    The introduction of \cite{DKPV97} sketches a proof for the Laplacian, omitting some technical details.   The body of the paper \cite{DKPV97} gives the full proof in full generality.  We present here the full proof but only for the second-order Laplace operator.  As with the other proofs, our hope is that by reducing the generality,
 we make this proof more accessible, as compared to the original source.

  Besides the three proofs presented in this paper, at least two more proofs of Dahlberg's square function estimates are available in the literature. One is, of course,
  Dahlberg's original proof \cite{D80} using the technique of good lambda inequalities combined with Dahlberg's estimates for harmonic measure on Lipschitz domains (\cite{D77}, \cite{D79}).
  Another proof is due to Fabes, Mendez, and M. Mitrea \cite{FMM98}.
  This proof uses layer potentials and Verchota's work \cite{V84} on the invertibility of layer potentials on Lipschitz domains.
  As we hope that these remarks make clear, Dahlberg's area integral estimates appear in many different contexts and are intertwined with many other estimates.
   The fact that their proof can be obtained from a number of different perspectives is a clear indication of their value.

  \section{Preliminary results from classical harmonic analysis}
    Denote $\RR_{+}^{n + 1} = \{(x, y) \in \RR^n \times \RR : y > 0\}$.
    For $x \in \RR^n$, $X = (x, 0) \in \partial \RR_{+}^{n + 1}$, define the non-tangential cone over $X$ by
    \begin{align*}
      \Gamma(X) \
      &= \ \left\{Y = (x', y) \in \RR^n \times \RR : |x' - x| < y \right\}.
    \end{align*}
    \begin{defn}
      For functions $w$ (real, complex, or vector-valued) defined on $\RR_{+}^{n + 1}$, the non-tangential maximal function of $w$ is defined by
      \begin{align*}
        N(w)(x) \
        &= \ \sup_{Y \in \Gamma(X)} |w(Y)|, \quad X = (x, 0).
      \end{align*}
    \end{defn}
    The Euclidean space analogue of theorems of Hardy and Littlewood is as follows (we will only need the case $p = 2$).
    \begin{thm}[{\cite[prop page 57]{S93}}]
      \label{thm:1.2}
      Consider the Poisson kernel
      \begin{align*}
        P_y(x) \
        &= \ \frac{c_n y}{(|x|^2 + y^2)^{n + 1}} \
        = \ (1/y^n) P_1(x/y), \quad x \in \RR^n, y > 0.
      \end{align*}
      Fix $1 < p \leq \infty$.
      If $f \in L^p(\RR^n)$, then the convolution
      \begin{align*}
        u(x, y) \
        &= \ P_y \ast f(x), \quad (x, y) \in \RR^n \times \RR_{+}
      \end{align*}
      is harmonic in the upper half space and has non-tangential limit equal to $f$:
      \begin{align*}
        \lim_{Y \to X} u(Y) \
        &= \ f(x), \text{ restricted to } Y \in \Gamma(X),
      \end{align*}
      for almost every $X = (x, 0) \in \RR^n \times \RR_{+}$ with respect to Lebesgue measure $\D x$.
      Moreover, 
      \begin{align*}
        \|N(u)\|_{L^p(\RR^n)} \
        &\leq \ C_{n, p} \|f\|_{L^p(\RR^n)}
      \end{align*}
      for a constant $C_{n, p}$ depending only on $n$ and $p$.
    \end{thm}
      
    A partial converse to Theorem \ref{thm:1.2} is as follows.
    \begin{thm}[{\cite[Theorem 1, p. 197; Corollary p. 200]{S70}}] \label{thm:1.3}
      Suppose that $u$ is harmonic in $\RR_{+}^{n + 1}$ and $1 < p \leq \infty$.
      If
      \begin{align*}
        \sup_{y > 0} \|u(\cdot, y)\|_{L^p(\RR^n)} < \infty,
      \end{align*}
      then $u = P_y \ast f$ for some $f \in L^p(\RR^n)$ and $u$ converges non-tangentially to $f$ almost everywhere $\D x$.
      Furthermore, 
      \begin{align*}
        \|f\|_{L^p(\RR^n)} \
        &= \ \sup_{y > 0} \|u(\cdot, y)\|_{L^p(\RR^n)}.
      \end{align*}
    \end{thm}

    We will be particularly interested in the special case of the upper half-plane ($n =1$). 
    The Hardy space $\cH^2(\RR_{+}^2)$ is the space of analytic functions $F$ such that
    \begin{align*}
      \|F\|_{\cH^2(\RR_{+}^2)} \
      &\coloneqq \ \sup_{y > 0} \|F(\cdot, y)\|_{L^2(\RR)} \
      < \ \infty.
    \end{align*}
    Since analytic functions are harmonic, Theorem \ref{thm:1.3} implies $F(z)$ is the convolution by the Poisson kernel of its (non-tangential) boundary values, denoted $F(x)$, $x\in \RR$.      Since convolution with the Poisson kernel is a contraction on $L^2(\RR)$, there is a Hilbert space isometry between the Hardy space and the boundary space $L^2(\RR)$\footnote{If we decompose $F$ into real and imaginary parts $F(z) = u(z) + \i v(z)$, then $v$ is the harmonic conjugate of $u$, and the boundary value $v(x)$ is the Hilbert transform of $u(x)$. Since the Hilbert transform is an isometry on $L^2(\RR)$, 
    \begin{align*}
      \frac{1}{2} \int\limits_{\RR} \left(u(x)^2 + v(x)^2 \right) \, \D x \
      &= \ \int\limits_{\RR} u(x)^2 \, \D x \
      = \ \int\limits_{\RR} v(x)^2 \, \D x.
    \end{align*}}
    \begin{align}\label{eq:Hardy}
      \|F\|_{\cH^2(\RR_{+}^2)}^2 \
      &= \ \int\limits_{\RR} |F(x)|^2 \, \D x \
      = \ \|F\|_{L^2(\RR)}.
    \end{align}
    Furthermore, by Theorem \ref{thm:1.2}, there is an absolute constant $C$ such that
    \begin{align}\label{eq:Hardy-max}
      \|N(F)\|_{L^2(\RR)} \
      &\leq \ C \|F\|_{L^2(\RR)}.
    \end{align}

  Returning to the case of general dimensions, consider  a cube $Q \subset \RR^n$ of side length $r$, and denote by $T(Q)$ the Carleson box in the upper half space
    \begin{align*}
      T(Q) \
      &= \ Q \times \hlint{0}{r} \subset \RR_{+}^{n + 1}.
    \end{align*}
    \begin{defn}
      A non-negative measure $\mu$ on the upper half space $\RR_{+}^{n + 1}$ is called a Carleson measure if its Carleson measure norm, $\|\mu\|_{CM}$, is finite:
      \begin{align*}
        \|\mu\|_{CM} \
        &\coloneqq \ \sup_{Q} \frac{\mu(T(Q))}{|Q|} < \infty, \quad (|Q| = r^n \text{ is the volume of the cube } Q).
      \end{align*}
    \end{defn}

    Carleson's theorem (\cite{Ca58}; see also \cite[Theorem 2, p. 59]{S93}) is as follows.
    \begin{thm}
      \label{thm:1.5}
      If $f$ is a measurable function on $\RR_{+}^{n + 1}$ and $\mu$ is a Carleson measure, then
      \begin{align*}
        \int\limits_{\RR_{+}^{n + 1}} |f|^2 \, \D \mu \
        &\leq \ C_n \|\mu\|_{CM} \int\limits_{\RR^n} N(f)^2 \, \D x,
      \end{align*}
      where $C_n$ is a dimensional constant.
    \end{thm}

    \begin{defn}
      The space of functions of bounded mean oscillation, $\BMO(\RR^n)$, is defined as the set of $f \in L_{loc}^1(\RR^n)$ whose BMO norm, $\|f\|_{\BMO}$, is finite:
      \begin{align*}
        \|f\|_{\BMO} \
        &\coloneqq \ \sup_{Q} \frac{1}{|Q|} \int\limits_{Q} \left|f(x) - f_Q \right| \, \D x \
        < \ \infty, \text{ with } f_Q \
        = \ \frac{1}{|Q|} \int\limits_{Q} f(x) \, \D x.
      \end{align*}
    \end{defn}
    Note that constant functions have BMO norm zero.
    The following relationship between BMO and Carleson measures was proved by C. Fefferman \cite[Theorem 3, p. 159]{S93}.
    
    \begin{thm}
      \label{thm:1.7}
      If $f \in \BMO(\RR^n)$ and $u(x, y) = P_y \ast f(x)$, then
      \begin{align*}
        \D \mu(x, y) \
        &= \ y |\nabla u(x, y)|^2 \, \D x \, \D y
      \end{align*}
      is a Carleson measure on $\RR_{+}^{n + 1}$ and 
      \begin{align*}
        \|\mu\|_{CM} \
        &\leq \ C_n \|f\|_{\BMO(\RR^n)}^2
      \end{align*}
      for some dimensional constant $C_n$.
    \end{thm}

    \begin{prop}[Green's formula]
      \label{prop:1.8}
      Suppose that $f$ is twice continuously differentiable in the closure of $\RR_{+}^{n + 1}$, and for some $\epsilon > 0$,
      \begin{align*}
        |f(x, y)| + y|\nabla f(x, y)| \
        &= \ O(|(x, y)|^{-n - \epsilon}), \quad \mbox{as} \ \ |(x,y)| \to \infty.
      \end{align*}
      Then
      \begin{align*}
        \int\limits_{\RR_{+}^{n + 1}} y \Delta f(x, y) \, \D x \, \D y \
        &= \ \int\limits_{\RR^n} f(x, 0) \, \D x.
      \end{align*}
      \begin{proof}
        Apply Green's formula in the half ball $\{|(x, y)| < R\} \cap \RR_{+}^{n + 1}$ to the integrand
        \begin{align*}
          y \Delta f \
          &= \ y \Delta f - f \Delta y,
        \end{align*}
        and take the limit as $R \to \infty$.
      \end{proof}
    \end{prop}

    We will now specialize to the case of planar domains, $n = 1$, and state some classical properties of conformal mappings.
    Consider a Lipschitz function
    \begin{align*}
      \varphi \colon \RR \to \RR, \quad |\varphi(\xi) - \varphi(\xi')| \
      \leq \ L |\xi - \xi'| \
      < \ \infty, \quad \xi, \xi' \in \RR,
    \end{align*}
    and the unbounded Lipschitz domain above its graph,
    \begin{align*}
      \Omega \
      &\coloneqq \ \left\{\zeta = \xi + \i \eta \in \CC : \eta > \varphi(\xi) \right\}.
    \end{align*}
    The upper half plane $\RR_{+}^2$ will be viewed as the complex upper half plane, that is,
    \begin{align*}
      \RR_{+}^2 \
      &= \ \left\{x + \i y : y > 0 \right\}.
    \end{align*}
    Fix $\zeta_0 \in \Omega$ and consider the unique conformal mapping 
    \begin{align*}
      \Phi \colon \RR_{+}^2 \to \Omega \text{ such that } \Phi(\i) = \zeta_0 \text{ and } \Phi(\infty) = \infty.
    \end{align*}
    $\Phi$ extends to a homeomorphism of the closure of $\RR_{+}^2$ to the closure of $\Omega$.
    $\Phi'$ has non-tangential limits for almost every $x \in \RR = \partial \RR_{+}^2$ and so does the inverse mapping $\Psi = \Phi^{-1}$ mapping $\Omega$ to $\RR_{+}^2$.
    $\Phi$ and $\Psi$ preserve angles in the interior, and sets of measure zero of $\RR$ correspond to sets of measure zero of $\partial \Omega$ (see, for example, \cite{Ke80}, \cite{Z68}).

    Here are the properties of the conformal mapping that will make it possible to transplant estimates on $\Omega$ to estimates on $\RR_{+}^2$.

    \begin{thm}
      \label{thm:1.9}
      If $\Phi$ maps $\RR_{+}^2$ to the region above a Lipschitz graph $\Omega$ with constant $L$, and $\Phi(\infty) = \infty$, then $\Phi(x), x \in \RR$ is absolutely continuous, and $\Phi'(x)$ equals the non-tangential limit of $\Phi'(z)$ almost everywhere $\D x$.
      Furthermore, 
      \begin{enumerate}[label=(\alph*)]
        \item The functions $\log \Phi'(z)$ and $\Phi'(z)^r, r \in \RR$, are well-defined (single-valued) on $\RR_{+}^2$.
        \item $|\arg \Phi'(z)| \leq \arctan L$ for all $z \in \RR_{+}^2$.
        \item The measure
          \begin{align*}
            \D \mu \
            &= \ y \frac{|\Phi''(z)|^2}{|\Phi'(z)|^2} \, \D x \, \D y \
            = \ y |(\log \Phi'(z))'|^2 \, \D x \, \D y
          \end{align*}
          is a Carleson measure with norm
          \begin{align*}
            \|\mu\|_{CM} \
            &\leq \ C \arctan L
          \end{align*}
          for some absolute constant $C$.
        \item The measure $w(x) \, \D x$ with $w(x) = |\Phi'(x)|, x \in \RR$ belongs to the Muckenhoupt class $A_2$,\footnote{For a discussion of the properties of $A_2$ measures, see \cite[p. 194]{S93}.} that is, for all intervals $I \subset \RR$,
          \begin{align*}
            \left(\frac{1}{|I|} \int_{I} w \, \D x \right) \left(\frac{1}{|I|} \int_{I} w^{-1} \, \D x \right) \
            &\leq \ C(L)
          \end{align*}
          for a constant $C(L)$ depending on $L$.
      \end{enumerate}
      \begin{proof}
        The proofs are in \cite[Section 1]{Ke80} with the exception of part (c).
        We repeat the main ideas here to keep this presentation self-contained.

        It is well-known that $\Phi$ restricted to the real axis is absolutely continuous and that $\Phi'(x)$ equals the non-tangential limit of $\Phi'(z)$ almost everywhere $\D x$. (For any $\zeta_1 \not\in \overline{\Omega}$, the mapping $\zeta \mapsto 1/(\zeta - \zeta_1)$ sends $\Omega$ to a bounded domain with rectifiable boundary. In the rectifiable case, these properties are proved in \cite[Theorems 10.15 and 10.17, pp. 293-294]{Z68}.)
        At every point where the derivative exists, $\Phi'(x)$ is the vector tangent to $\partial \Omega$.
        Hence,
        \begin{align*}
          \left|\arg \Phi'(x) \right| \
          &\leq \ \arctan L \
          < \ \pi/2 \ \text{ a.e. } \D x.
        \end{align*}
        Moreover, $\arg \Phi'(z)$ is the Poisson integral of its boundary values and satisfies the maximum principle:
        \begin{align*}
          \left|\arg \Phi'(z) \right| \
          &\leq \ \arctan L, \quad z \in \RR_{+}^2.
        \end{align*}
        In particular, $\Phi'(z) \neq 0$, and $\log \Phi'(z)$ and $\Phi'(z)^r, r \in \RR$, are well-defined and analytic for $z \in \RR_{+}^2$.

        To prove part (c), set
        \begin{align*}
          g(z) \
          &= \ \log |\Phi'(z)|, \quad h(z) \
          = \ \arg \Phi'(z).
        \end{align*}
        Then $g$ restricted to $\RR$ is the Hilbert transform of $-h$ restricted to $\RR$.
        The Hilbert transform is bounded from $L^{\infty}(\RR)$ to $\BMO(\RR)$ \cite[6.16(b), p. 164]{S70}, so there is an absolute constant $C$ for which 
        \begin{align*}
          \|g\|_{\BMO(\RR)} \
          &\leq \ C \|h\|_{L^{\infty}(\RR)} \
          \leq \ C \arctan L.
        \end{align*}
        The measure
        \begin{align*}
          \D \mu \
          &= \ y \frac{|\Phi''(z)|^2}{|\Phi'(z)|^2} \, \D x \, \D y \
          = \ y \left|(\log \Phi'(z))' \right|^2 \, \D x \, \D y \
          = \ y |\nabla g(z)|^2 \, \D x \, \D y
        \end{align*}
        is a Carleson measure with norm
        \begin{align*}
          \|\mu\|_{CM} \
          &\leq \ C \|g\|_{\BMO(\RR)}^2 \
          \leq \ C^2 \|h\|_{L^{\infty}(\RR)}^2 \
          = \ C^2 \left(\arctan L \right)^2
        \end{align*}
        for some absolute constant $C$.
        The first inequality is Theorem \ref{thm:1.7}.

        To prove part (d), we follow the argument in A. P. Calder\'on's breakthrough paper on the Cauchy integral on Lipschitz curves \cite{C77}.
        As in that paper, using an approximation argument, one can assume without loss of generality that $\varphi \in C_0^{\infty}(\RR)$.
        Then $\Phi$ and $\Phi'$ extend continuously to the closed upper half plane $\re(z) \geq 0$.
        Moreover, $\Phi'$ is nonzero and bounded on $\re(z) \geq 0$, and the vector $\Phi'(x)$ is tangent to the curve $\partial \Omega$ at $\Phi(x)$.
        By part (b),
        \begin{align*}
          \left|\arg \Phi'(z) \right| \
          &\leq \ \arctan L \text{ for all } z, \re(z) \geq 0.
        \end{align*}
        It follows that
        \begin{align*}
          0 \
          &\leq \ \re (\Phi'(z)) \
          \leq \ |\Phi'(z)| \
          \leq \ \re (\Phi'(z)) (1 + L^2)^{1/2} \text{ for all } z, \re(z) \geq 0.
        \end{align*}
        Let $I \subset \RR$ be a finite interval.
        Comparing with the Poisson kernel, we have
        \begin{align*}
          \frac{1}{|I|} \int\limits_{I} \re(\Phi'(x)) \, \D x \
          &\leq \ \pi \re\left(\Phi'(x_I + \i |I|/2) \right),
        \end{align*}
        where $x_I$ is the midpoint of the interval $I$.
        Hence,
        \begin{align*}
          \frac{1}{|I|} \int\limits_{I} |\Phi'(x)| \, \D x \
          &\leq \ \pi \re\left(\Phi(x_I + \i |I|/2) \right) \left(1 + L^2 \right)^{1/2}.
        \end{align*}
        The same inequality holds for $(\Phi')^{-1}$, and so
        \begin{align*}
          \frac{1}{|I|} \int\limits_{I} |\Phi'(x)| \, \D x \ \frac{1}{|I|} \int\limits_{I} |\Phi'(x)|^{-1} \, \D x \
          &\leq \ \pi^2 |I|^2 (1 + L^2).
        \end{align*}
        This concludes the proof of part (d).
      \end{proof}
    \end{thm}

    One more important property of $\Phi$ that we will need is an estimate relating the distance to the boundary in $\RR_{+}^2$ and $\Omega$.
    Denote by $\delta \colon \Omega \to \RR$ (or $\delta_{\Omega}$) the distance to the boundary
    \begin{align*}
      \delta(\zeta) \
      &= \ \min_{\zeta' \in \partial \Omega} |\zeta - \zeta'|.
    \end{align*}

    \begin{defn}
      We will write
      \begin{align*}
        A \
        &\lesssim \ B
      \end{align*}
      to mean there exists a constant $C$ depending only on $L$ such that
      \begin{align*}
        A \
        &\leq \ C B.
      \end{align*}
      Similarly for $A \gtrsim B$.
      And $A \simeq B$ means both $A \lesssim B$ and $A \gtrsim B$.
    \end{defn}

    The estimates we need for the mapping $\Phi$ are as follows \cite{JK82}, \cite[p. 22]{P75}, \cite[p. 98]{LV73}.

    \begin{prop}
      \label{prop:1.11}
      If $\Phi \colon \RR_{+}^2 \to \Omega$ then 
      \begin{enumerate}[label=(\alph*)]
        \item $\Phi$ extends to a global quasiconformal mapping $\Psi$ of $\RR^2$ onto $\RR^2$, that is, a mapping such that for any $z_1, z_2, z_3 \in \CC$,
          \begin{align*}
            |z_1 - z_2| \simeq |z_2 - z_3| \
            &\iff \ |\Psi(z_1) - \Psi(z_2)| \simeq |\Psi(z_2) - \Psi(z_3)|.
          \end{align*}
          In particular, $\Phi$ and $\Phi^{-1}$ send non-tangential disks to regions comparable to non-tangential disks.
          Namely, if $z = x + \i y \in \RR_{+}^2$, then (with $\delta = \delta_{\Omega}$)
          \begin{align*}
            \left|\zeta - \zeta' \right| \
            &\lesssim \
            \delta(\zeta) \ \simeq \ \delta(\zeta') \text{ for all } \zeta, \zeta' \in \Phi(B(z, y/2)).
          \end{align*}
          Similarly, $|z' - z| \lesssim \im(z) \simeq \im(z')$ for all $z, z' \in \Phi^{-1}(B(\zeta, \delta(\zeta)/2))$ and all $\zeta \in \Omega$.
        \item If $z = x + \i y \in \RR_{+}^2$, then
          \begin{align*}
            \delta(\Phi(z)) \
            &\simeq \ y |\Phi'(x + \i y)|.
          \end{align*}
      \end{enumerate}
    \end{prop}

    Note that under the scaling $\tilde{\Phi} = a \Phi$, $\tilde{\Omega} = a \Omega$, $\tilde{\zeta}_0 = a \zeta_0$, the left and right sides of the expressions in Proposition \ref{prop:1.11} are multiplied by the factor $a$ and the Lipschitz constant is preserved.
    So it is natural that the constants in this comparison do not depend on the choice of $\zeta_0$.

    \begin{proof}
      The proof that $\Phi$ has a quasiconformal extension follows from a paper of Beurling-Ahlfors \cite{BA56}.
      That paper shows that the conformal mapping $\Phi$ extends from $\RR_{+}^2$ to a global quasiconformal mapping of $\RR^2$ if and only if the measure on the real line
      \begin{align*}
        \omega(E) \
        &= \ \int\limits_{E} w(x) \, \D x, \quad E \subset \RR
      \end{align*}
      satisfies the so-called quasi-symmetry condition
      \begin{align}
        \omega(I) \
        &\lesssim \ \omega(I')
        \label{eq:11typed}
      \end{align}
      for all adjacent intervals $I, I' \subset \RR$ of equal length.

      Inequality \eqref{eq:11typed} follows easily from Theorem \ref{thm:1.9} (d) as we now demonstrate.
      The Cauchy-Schwarz inequality implies
      \begin{align*}
        \left(\int_I \, \D x \right)^2 \
        &\leq \ \left(\int_I w \, \D x \right) \left(\int_{I} w^{-1} \, \D x \right),
      \end{align*}
      which can be rewritten as
      \begin{align*}
        \frac{1}{\omega(I)} \
        &\leq \ \frac{1}{|I|^2} \int\limits_{I} w^{-1} \, \D x.
      \end{align*}
      Suppose that $I$ and $I'$ are adjacent intervals of equal length and $J = I \cup I'$.
      Then
      \begin{align*}
        \frac{\omega(J)}{\omega(I)} \
        &\leq \ \omega(J) \frac{1}{|I|^2} \int\limits_{I} w^{-1} \, \D x \
        = \ \left(\frac{2}{|J|} \int_{J} w \, \D x \right) \left(\frac{2}{|J|} \int_I w^{-1} \, \D x \right) \
        \leq \ 4 C(L),
      \end{align*}
      with $C(L)$ the constant in Theorem \ref{thm:1.9} (d).
      This ends the proof of part (a).

      Part (a) says in particular that $\Phi$ maps non-tangential disks to regions comparable to non-tangential disks.
      Thus, the Schwarz lemma applied to $\Phi$ and $\Phi^{-1}$ 
      implies $|\Phi'(z)| \simeq \delta(\Phi(z))/y$ for all $z \in \RR_{+}^2$.
Hence,  (b) holds.
    \end{proof}

    \begin{rmk}
      \label{rmk:1.12}
      A non-tangential approach region for a boundary point $z \in \partial \Omega$ is given by
      \begin{align*}
        \Gamma_{\alpha}(z, \Omega) \
        &= \ \left\{z' \in \Omega : |z' - z| \leq (1 + \alpha) \delta_{\Omega}(z') \right\}.
      \end{align*}
      It follows from Proposition \ref{prop:1.11} (a) that there are positive constants $\beta$ and $\gamma$ depending only on $\alpha$ and the Lipschitz constant of $\Omega$ such that for any $x \in \RR$,
      \begin{align*}
        \Gamma_{\beta}(\Phi(x), \Omega) \
        &\subset \ \Phi\left(\Gamma_{\alpha}(x, \RR_{+}^2) \right) \
        \subset \ \Gamma_{\gamma}(\Phi(x), \Omega).
      \end{align*}
    \end{rmk}

  \section{Analytic and harmonic functions in dimension two}
    \begin{thm}
      \label{thm:2.1}
      There is an absolute constant $\epsilon_0 > 0$ such that the following holds.
      Suppose that $F$ is an analytic function in $\Omega$, the region above a Lipschitz graph with Lipschitz constant $L \leq \epsilon_0$.
      Suppose that
      \begin{align*}
        \int\limits_{\Omega} \delta(\zeta) |F'(\zeta)|^2 \, \D \xi \, \D \eta \
        &= \ A \
        < \ \infty, \quad \zeta = \xi + \i \eta.
      \end{align*}
      Assume further that for any $\delta_0 > 0$,
      \begin{align*}
        F(\zeta) \to 0 \text{ as } |\zeta| \to \infty \text{ in } \delta(\zeta) \geq \delta_0.
      \end{align*}
      Then $F$ has non-tangential limit $F_0(\zeta)$ for almost every $\zeta \in \partial \Omega$ with respect to arc length $\D \sigma$ and 
      \begin{align*}
        \int\limits_{\partial \Omega} |F_0(\zeta)|^2 \, \D \sigma(\zeta) \
        &\leq \ C A
      \end{align*}
      for some absolute constant $C$.
    \end{thm}

    \begin{rmk}
      We will give a second proof of Theorem \ref{thm:2.1}, valid for all $L$, rather than sufficiently small $L$ in the Appendix.
      That proof will not be entirely self-contained as it relies on the theory of weighted norm inequalities.
      In Section 3, we give a different proof of this estimate that is valid in all dimensions and does not assume $L \leq \epsilon_0$.
    \end{rmk}

    \begin{proof}
      We begin with an explanation of the idea of the proof in the special case in which $\partial \Omega$ is smooth and $F$ is smooth and has suitable decay at infinity.
      Then we will use an approximation argument to give a complete proof.

      For $z = x + \i y \in \RR_{+}^2$, define the analytic function
      \begin{align*}
        H(z) \
        &= \ F \circ \Phi(z) (\Phi'(z))^{1/2}.
      \end{align*}
      For $\partial \Omega$ smooth, $\Phi$ is smooth, and the change of variables $\zeta = \Phi(z)$ yields $\D \sigma(\zeta) = |\Phi'(z)| \, \D x$ and 
      \begin{align*}
        \int\limits_{\partial \Omega} |F(\zeta)|^2 \, \D \sigma(\zeta) \
        &= \ \int\limits_{\RR} |H(x)|^2 \, \D x.
      \end{align*}
      Since $\Delta |H(z)|^2 = 4 |H'(z)|^2$, if $H$ decays suitably at infinity, then Green's formula (Proposition \ref{prop:1.8}) applied to $f = |H|^2$) gives
      \begin{align*}
        \int\limits_{\RR} |H(x)|^2 \, \D x \
        &= \ 4 \int\limits_{\RR_{+}^2} y |H'(z)| \, \D x \, \D y.
      \end{align*}
      Next, we write
      \begin{align*}
        H'(z) \
        &= \ \left(F' \circ \Phi(z) \right) \left(\Phi'(z) \right)^{3/2} + \frac{1}{2} \left(F \circ \Phi(z) \right) \left(\Phi'(z) \right)^{-1/2} \Phi''(z) \\
        &= \ M(z) + E(z).
      \end{align*}
      The first term, $M(z)$, is controlled by $A$.
      Indeed, using Proposition \ref{prop:1.11},
      \begin{align*}
        A \
        &= \ \int\limits_{\Omega} \delta(\zeta) |F'(\zeta)|^2 \, \D \xi \, \D \eta \
        = \ \int\limits_{\RR_{+}^2} \delta(\Phi(z)) |F' \circ \Phi(z)|^2 |\Phi'(z)|^2 \, \D x \, \D y \\
        &\simeq \ \int\limits_{\RR_{+}^2} y |\Phi'(z)| \, |F' \circ \Phi(z)|^2 |\Phi'(z)|^2 \, \D x \, \D y \\
        &= \ \int\limits_{\RR_{+}^2} y |M(z)|^2 \, \D x \, \D y.
      \end{align*}
      In other words,
      \begin{align*}
        \int\limits_{\RR_{+}^2} y |M(z)|^2 \, \D x \, \D y \
        &\lesssim \ A.
      \end{align*}
      To estimate the error term $E(z)$, we write
      \begin{align*}
        \int\limits_{\RR_{+}^2} y |E(z)|^2 \, \D x \, \D y \
        &= \ \frac{1}{4} \int\limits_{\RR_{+}^2} y |F \circ \Phi(z)|^2 \frac{|\Phi''(z)|^2}{|\Phi'(z)|} \, \D x \, \D y \
        = \ \int\limits_{\RR_{+}^2} |H(z)|^2 \, \D \mu(z)
      \end{align*}
      with
      \begin{align*}
        \D \mu(z) \
        &= \ y \frac{|\Phi''(z)|^2}{|\Phi'(z)|^2} \, \D x \, \D y.
      \end{align*}
      It follows from Theorem \ref{thm:1.9} (c) and Theorems \ref{thm:1.5} and \ref{thm:1.2} that
      \begin{align*}
        \int\limits_{\RR_{+}^2} |H(z)|^2 \, \D \mu(z) \
        &\lesssim \ \|\mu\|_{CM} \|N(H)\|_{L^2(\RR)} \
        \lesssim \ \epsilon_0^2 \int\limits_{\RR} |H(x)|^2 \, \D x.
      \end{align*}
      In all, there is an absolute constant $C$ such that
      \begin{align*}
        \int\limits_{\RR} |H(x)|^2 \, \D x \
        &\leq \ C A + C \epsilon_0^2 \int\limits_{\RR} |H(x)|^2 \, \D x.
      \end{align*}
      Assuming the integral on both sides is finite and that $C \epsilon_0^2 \leq 1/2$, we can subtract and obtain
      \begin{align*}
        \frac{1}{2} \int\limits_{\RR} |H(x)|^2 \, \D x \
        &\leq \ C A,
      \end{align*}
      ending the proof of Theorem \ref{thm:2.1} under the extra assumptions.

      This proof is incomplete in two ways.
      First of all, the function $H$ need not be smooth up to the boundary $\partial \RR_{+}^2$ and need not decay at $\infty$ fast enough to validate the step using Green's formula.
      Secondly, we need the $L^2(\RR)$ integral of $H$ to be finite in order to be allowed to subtract the term with the factor $C \epsilon_0^2$.

      To take care of these issues, we modify $H$ by defining
      \begin{align*}
        H_{\gamma}(z) \
        &= \ \frac{H(z)}{\i + \gamma z} \
        = \ \frac{F \circ \Phi(z) \Phi'(z)^{1/2}}{\i + \gamma z}, \quad \gamma > 0.
      \end{align*}
      For all $\gamma > 0$ and all $k \geq 1$, we will prove the qualitative bound
      \begin{align}
        \sup_{y \geq 1/k} \int\limits_{\RR} |H_{\gamma}(x + \i y)|^2 \, \D x \
        &< \ \infty.
        \label{eq:12typed}
      \end{align}
      As before, we then show that there is an absolute constant $C$ such that 
      \begin{align}
        \int\limits_{\RR} |H_{\gamma}(x + \i/k)|^2 \, \D x \
        &\leq \ C A + C \epsilon_0^2 \int\limits_{\RR} |H_{\gamma}(x + \i/k)|^2 \, \D x \text{ for } k \geq 1.
        \label{eq:13typed}
      \end{align}
      (But there will also be a new term arising from the denominator $\i + \gamma z$.)
      Choosing $\epsilon_0$ so that $C \epsilon_0^2 \leq 1/2$, we obtain
      \begin{align*}
        \frac{1}{2} \int\limits_{\RR} |H_{\gamma}(x + \i/k)|^2 \, \D x \
        &\leq \ C A.
      \end{align*}
      The theorem will follow by taking the limit first as $\gamma \to 0$ and then as $k \to \infty$.

      To obtain \eqref{eq:12typed}, we first show that
      \begin{align}
        \sup_{y \geq 1/k} |H(z)|^2 \
        &= \ \sup_{y \geq 1/k} |F \circ \Phi(z)|^2 |\Phi'(z)| \
        \lesssim \ A k.
        \label{eq:14typed}
      \end{align}
      Indeed, since $|F'|^2$ is subharmonic,
      for $\zeta \in \Omega$, $\delta = \delta(\zeta)$ we have
      \begin{align*}
        |F'(\zeta)|^2 \
        &\leq \ \frac{1}{|B_{\delta/2}|} \int\limits_{B_{\delta/2}(\zeta)} |F'(\zeta')|^2 \, \D \xi' \, \D \eta' \\
        &\leq \ \frac{2/\delta}{|B_{\delta/2}|} \int\limits_{B_{\delta/2}(\zeta)} \delta(\zeta') |F'(\zeta')|^2 \, \D \xi' \, \D \eta' \\
        &\leq \ \frac{10 A}{\delta(\zeta)^3}.
      \end{align*}
      Next, if the Lipschitz constant $L \leq 1$, then $\delta(\zeta + \i s) \simeq \delta(\zeta) + s$.
      Thus, the fact that $F(\zeta + \i s) \to 0$ as $s \to \infty$ implies
      \begin{align*}
        |F(\zeta)| \
        &\leq \ \int\limits_{0}^{\infty} |F'(\zeta + \i s)| \, \D s \
        \leq \ C A^{1/2} \int\limits_{0}^{\infty} \left(\delta(\zeta) + s \right)^{-3/2} \, \D s \
        \lesssim \ \frac{A^{1/2}}{\delta(\zeta)^{1/2}}.
      \end{align*}
      For $z = x + \i y$, $\zeta = \Phi(z)$, Proposition \ref{prop:1.11} says $\delta(\zeta) \simeq y |\Phi'(z)|$, so we get
      \begin{align}
        |F(\Phi(z))|^2 |\Phi'(z)| \
        &\lesssim \ \frac{A}{y},
        \label{eq:15typed}
      \end{align}
      and this immediately implies \eqref{eq:14typed}.

      Next, the bound \eqref{eq:14typed} yields
      \begin{align*}
        \int\limits_{\RR} |H_{\gamma}(x + \i y)|^2 \, \D x \
        &\lesssim \ \int\limits_{\RR} \frac{A k}{1 + \gamma^2 x^2} \, \D x \
        \lesssim \ \frac{Ak}{\gamma} \text{ for all } y \geq 1/k,
      \end{align*}
      and this proves \eqref{eq:12typed}.

      Another immediate consequence of \eqref{eq:14typed} is the pointwise bound
      \begin{align*}
        |H_{\gamma}(z)| \
        &\lesssim \ \frac{A k}{1 + \gamma |z|}, \quad y \geq 1/k.
      \end{align*}
      The Cauchy integral formula implies that for an analytic function $h$ in a disk,
      \begin{align*}
        |h'(0)| \
        &\leq \ \frac{1}{r} \max_{|z| \leq r} |h(z)|.
      \end{align*}
      It follows, using the disk of radius $y/2$ centered at $z$, that
      \begin{align*}
        y |H_{\gamma}'(z)| \
        &\lesssim \ \frac{A k}{1 + \gamma |z|}, \quad y \geq 1/k.
      \end{align*}
      Assembling these estimates, we see that the function $f(x, y) = |H_{\gamma}(z + \i/k)|^2$ is smooth in $\overline{\RR_{+}^2}$ and satisfies the estimate
      \begin{align*}
        |f(x, y)| + y |\nabla f(x, y)| \
        &= \ |H_{\gamma}(z + \i/k)|^2 + 2 y |H_{\gamma}'(z + \i/k)| \, |H_{\gamma}(z + \i/k)| \\
        &\lesssim \ \frac{A^2 k^2}{(1 + \gamma |z|)^2}.
      \end{align*}
      In particular, $|f| + y |\nabla f| = O(|(x, y)|^{-2})$ as $|(x, y)| \to \infty$, and therefore we can apply Green's formula (Proposition \ref{prop:1.8}).
      This gives
      \begin{align*}
        \int\limits_{\RR} |H_{\gamma}(x + \i/k)|^2 \, \D x \
        &= \ 4 \int\limits_{\RR_{+}^2} y |H_{\gamma}'(z + \i/k)|^2 \, \D x \, \D y.
      \end{align*}
      The formula for $H_{\gamma}'$ is
      \begin{align*}
        H_{\gamma}'(z) \
        &= \ M_{\gamma}(z) + E_{\gamma}(z) + \tilde{E}_{\gamma}(z),
      \end{align*}
      in which
      \begin{align*}
        M_{\gamma} \
        &= \ \frac{(F' \circ \Phi) (\Phi')^{3/2}}{\i + \gamma z}; \quad E_{\gamma} \
        = \ \frac{H_{\gamma} \Phi''}{2 \Phi'}; \quad \tilde{E}_{\gamma} \
        = \ -\frac{\gamma (F \circ \Phi) (\Phi')^{1/2}}{(\i + \gamma z)^2}.
      \end{align*}
      The first term is controlled by the estimate we already have:
      \begin{align*}
        \int\limits_{\RR_{+}^2} y |M_{\gamma}(z + \i/k)|^2 \, \D x \,\D y \
        &\leq \ \int\limits_{\RR_{+}^2} y |M(z)|^2 \, \D x \, \D y \
        \lesssim \ A.
      \end{align*}
      For the term $E_{\gamma}$, we write
      \begin{align*}
        \int\limits_{\RR_{+}^2} y |E_{\gamma}(z + \i/k)|^2 \, \D x \, \D y \
        &= \ \int\limits_{\RR_{+}^2} |H_{\gamma}(z + \i/k)|^2 \, \D \mu_k(z)
      \end{align*}
      with 
      \begin{align*}
        \D \mu_k(z) \
        &= \ y \frac{|\Phi''(z + \i/k)|^2}{4 |\Phi'(z + \i/k)|^2} \, \D x \, \D y.
      \end{align*}
      Since $|\arg \Phi'(x + \i/k)| \leq \arctan L$, we can use Theorem \ref{thm:1.9} (c) as in the special case to show that $\|\mu_k\|_{CM} \leq C (\arctan L)^2 \lesssim \epsilon_0^2$.
      Hence, by Theorems \ref{thm:1.2} and \ref{thm:1.5},
      \begin{align*}
        \int\limits_{\RR_{+}^2} |H_{\gamma}(z + \i/k)|^2 \, \D \mu_k(z) \
        &\lesssim \ \epsilon_0^2 \int\limits_{\RR} |H_{\gamma}(x + \i/k)|^2 \, \D x.
      \end{align*}

      Next, we need to estimate the term involving $\tilde{E}_{\gamma}$.
      Applying the estimate \eqref{eq:15typed}, we have
      \begin{align*}
        |F(\Phi(z + \i/k))|^2 |\Phi'(z + \i/k)| \
        &\lesssim \ \frac{A}{y + 1/k}.
      \end{align*}
      Since $y/(y + 1/k) \leq 1$ and $|\i + \gamma(z + \i/k)|^2 \geq 1 + \gamma^2 |z|^2$, it follows that
      \begin{align*}
        \int\limits_{\RR_{+}^2} y |\tilde{E}_{\gamma}(z + \i/k)|^2 \, \D x \, \D y \
        &\lesssim \ \int\limits_{\RR_{+}^2} \frac{A \gamma^2}{(1 + \gamma^2 |z|^2)^2} \, \D x \, \D y \\
        &= \ A \int\limits_{\RR_{+}^2} \frac{1}{(1 + |z'|^2)^2} \, \D x' \, \D y' \
        = \ A C,
      \end{align*}
      with $z' = \gamma z$ and $C$ an absolute constant.

      Combining the bounds for all three terms, there is an absolute constant $C$ for which
      \begin{align*}
        &\int\limits_{\RR} \left|H_{\gamma}(x + \i/k) \right|^2 \, \D x \\
        &\qquad \qquad = \ 4 \int\limits_{\RR_{+}^2} y \left|H_\gamma'(z + \i/k) \right|^2 \, \D x \, \D y \\
        &\qquad \qquad \leq \ 12 \int\limits_{\RR_{+}^2} y \left(|M_{\gamma}(z + \i/k)|^2 + |E_{\gamma}(z + \i/k)|^2 + |\tilde{E}_{\gamma}(z + \i/k)|^2 \right) \, \D x \, \D y \\
        &\qquad \qquad \leq \ C A + C \epsilon_0^2 \int\limits_{\RR} \left|H_{\gamma}(x + \i/k) \right|^2 \, \D x.
      \end{align*}
      Choosing $C \epsilon_0^2 \leq 1/2$, we obtain
      \begin{align}
        \int\limits_{\RR} \left|H_{\gamma}(x + \i/k) \right|^2 \, \D x \
        &\lesssim \ A
      \end{align}
      with a constant independent of $\gamma > 0$ and $k > 0$.

      For each fixed $k$, Fatou's lemma implies that
      \begin{align*}
        \int\limits_{\RR} \left|H_0(x + \i/k) \right|^2 \, \D x \
        &\leq \ \liminf_{\gamma \to 0} \int\limits_{\RR} \left|H_{\gamma}(x + \i/k) \right|^2 \, \D x \
        \lesssim \ A.
      \end{align*}
      Since the bounds are independent of $k$, $H$ belongs the Hardy space $\cH^2(\RR^2_+) $ with non-tangential limits, which we will denote $H_0(x)$ almost everywhere $\D x$ (see \eqref{eq:Hardy} and \eqref{eq:Hardy-max}).    Moreover, $\Phi'$ has {\em non-zero} non-tangential limits almost everywhere $\D x$, so $F \circ \Phi = H/(\Phi')^{1/2}$ also has non-tangential limits almost everywhere $\D x$ (see Remark \ref{rmk:1.12}).
      Furtherrmore, $\Phi$ maps sets of measure zero $\D x$ to sets of measure zero $\D \sigma$ and vice versa.
      Thus,
      \begin{align*}
        F(\zeta') \to F_0(\zeta) \text{ as } \zeta' \to \zeta = \Phi(x) \text{ non-tangentially for } \zeta' \in \Omega
      \end{align*}
      for almost every $\zeta = \Phi(x) \in \partial \Omega$ with respect to $\D \sigma$,
      where we define $F_0$ by 
            \begin{align*}
        F_0(\zeta) \
        & : \, =  \ H_0(x) \Phi'(x)^{-1/2}, \quad \zeta \ = \ \Phi(x).
      \end{align*}
          Finally,
      \begin{align*}
        \int\limits_{\partial \Omega} |F_0(\zeta)|^2 \, \D \sigma(\zeta) \
        &= \ \int\limits_{\RR} \left|H_0(x) \Phi'(x)^{-1/2} \right|^2 |\Phi'(x)| \, \D x \
        = \ \int\limits_{\RR} \left|H_0(x) \right|^2 \, \D x \
        \lesssim \ A.
      \end{align*}
      This concludes the proof of Theorem \ref{thm:2.1}.
    \end{proof}
  As mentioned in the introduction, we will present a second proof of Theorem 2.1 in the plane, valid for all $L$.
    To avoid disrupting the flow of the exposition, we will present this proof in the Appendix (see Theorem \ref{thm:A.1}). 

    \begin{cor}
      \label{cor:2.2written}
      Let $\Omega$ and $\epsilon_0 > 0$ be as in Theorem \ref{thm:2.1}.
    Suppose that $F$ is analytic in $\Omega$, and 
      \[
      \int_{\Omega} \delta(\zeta) |F'(\zeta)|^2 \, \D \xi \, \D \eta = A < \infty, 
      \quad (\zeta = \xi + \i \eta).
      \]
      Then there exists $c\in\CC$ such that for any $\delta_0 > 0$,
      \begin{align*}
        F(\zeta) \to c \ \text{ as } |\zeta| \to \infty \ \ \text{for} \ \ \zeta\in \Omega \ \ \mbox{satisfying} \ \ \delta(\zeta) \geq \delta_0.
      \end{align*}
      Consequently, by Theorem \ref{thm:2.1}, 
      $F - c$ has non-tangential limit $F_0(\zeta) - c$ for almost every $\zeta \in \partial \Omega$, with respect to arc length $\D \sigma$, and an absolute constant $C$ such that 
      \begin{align*}
        \int\limits_{\partial \Omega} \left|F_0(\zeta) - c \right|^2 \, \D \sigma(\zeta) \
        &\leq \ C A.
      \end{align*}
    \end{cor}
    To prove the corollary it suffices to show that the constant $c$ exists.  This proof is
    straightforward and is likewise deferred to the Appendix to streamline the presentation (see Proposition \ref{prop:A.2}).

  \section{Endpoint counterexample for the inhomogeneous Dirichlet problem}
    In this section and Section 4, we will be using notations for $L^2$-Sobolev spaces that are consistent with Sections 2-4 of \cite{JK95}.
    Namely,
    \begin{align*}
      L_{\alpha}(\Omega) \
      &= \ H^{\alpha}(\Omega), \quad \dot{L}_{\alpha}(\Omega) \
      = \ \dot{H}^{\alpha}(\Omega), \text{ and } \\
      L_1^2(\partial \Omega) \
      &= \ \left\{f \in L^2(\partial \Omega, \D \sigma) : \nabla_T f \in L^2(\partial \Omega, \D \sigma) \right\},
    \end{align*}
    with $\nabla_T f$ the tangential gradient.
    The main result of this section, Theorem \ref{thm:3.1}, shows how to deduce Theorem \ref{thm:1.2} (c) for $p = 2$ in \cite{JK95} from Corollary \ref{cor:2.2written}.

    \begin{thm}
      \label{thm:3.1}
      For any $\epsilon > 0$, there exists a Lipschitz function $\varphi \colon \RR \to \RR$ such that $\|\varphi\|_{\infty} \leq \epsilon$, and a compactly supported function $g \in L_{-1/2}^2(\Omega)$ with $\Omega = \{(x, y) \in \RR^2 : y > \varphi(x)\}$, such that the solution to $\Delta u = g$ with $\Tr(u) = 0$ on $\partial \Omega$, is such that $u \not\in \dot{L}_{3/2}^2(\Omega)$.
      \begin{proof}
        The starting point is the example due to Guy David (unpublished) of a function $f \in \dot{L}_{3/2}^2(\Omega)$ such that $\Tr(f) \not\in L_1^2(\partial \Omega)$.
        We define the zig-zag function $\theta$ by $\theta(x) = \epsilon x$ for $0 \leq x \leq 1$, $\theta(x) = 2 \epsilon - \epsilon x$ for $1 \leq x \leq 2$, and $\theta(x + 2) = \theta(x)$.
        We then let $\varphi(x) = 2^{-4^N} \theta(2^{4^N} x)$ for $2^{-N} \leq x \leq 2^{-N + 1}$, $N = 0, 1, 2, \ldots$, $\varphi(x) = 0$ for $|x| \geq 2$ (for more details, see \cite[p. 176]{JK95}).
        We let $f'(y) = (\log 10/|y|)^{-1/2} \alpha(y)$, where $\alpha$ is a smooth cut-off function.
        It is not difficult to see that $f'(y) \in L_{1/2}^2(\RR)$, $f'(0) = \infty$, and that $h(x) = f(\varphi(x))$ satisfies $|h'(x)| \not\in L^2([-3,3])$.
        Finally, let the cut-off function $\beta\in C_0^\infty((-4,4))$ be equal to $1$ in a neighborhood of $[-3,3]$.   If we define $w(x, y) = \beta(x) f(y)$, then $w \in \dot{L}_{3/2}^2(\RR^2)$, but $\Tr(w)$ on the graph of $\varphi$ does not belong to $\dot{L}_{1}^2(\partial \Omega)$.
        
        Let $u$ solve $\Delta u = g$, $\Tr(u) = 0$, in $\Omega$, the domain above the graph of $\varphi$, where $g = \Delta w$, so that $g \in \dot{L}_{-1/2}^2(\Omega)$.
        Assume, for contradiction, that $u \in \dot{L}_{3/2}^2(\Omega)$, and let $v = u - w$.
        Then $v$ is harmonic in $\Omega$, $v \in \dot{L}_{3/2}^{2}(\Omega)$, and $\Tr(v)$ on $\partial \Omega$ equals $-\Tr(w)$, which does not belong to $\dot{L}_1^2(\partial \Omega)$.
        
        Consider now the analytic function $F = \p_y v + \i \p_x v$.
        Since $v \in \dot{L}_{3/2}^2(\Omega)$, $F \in \dot{L}_{1/2}^{2}(\Omega)$.
        By the analyticity of $F$, the elementary argument at the bottom of page 182 of \cite{JK95} gives that $\int_{\Omega} \delta(\zeta) |F'(\zeta)|^2 \, \D \zeta < \infty$.
By Corollary \ref{cor:2.2written}, we conclude that for some $c \in \CC$, $F_0 - c$ is in $L^2(\partial \Omega, \D \sigma)$, where $F_0$ is the non-tangential limit of $F$.
        By the argument carried out in full in $n$ dimensions in 
        Corollary \ref{cor:4.7} below,  $\Tr(v) \in \dot{L}_1^2(\partial \Omega)$, which contradicts the previous conclusion that $\Tr(v) \notin \dot{L}_1^2(\partial \Omega)$.
                      \end{proof}
    \end{thm}

  \section{Estimates in all dimensions}
    In this section, we prove Theorem \ref{thm:2.1} for gradients of harmonic functions in the setting of bounded Lipschitz domains in all dimensions and with no restrictions on the Lipschitz constant.
    Recall that for an open set $\Omega$,
    \begin{align*}
      \delta_{\Omega}(X) \
      &= \ \min_{Q \in \partial \Omega} |X - Q|.
    \end{align*}

    \begin{thm}[Dahlberg \cite{D80}]
      \label{thm:4.1}
      Let $\Omega \subset \RR^{n + 1}$, $n = 1, 2, \ldots$ be a bounded Lipschitz domain.
      There exist constants $C_1$ and $C_2$ depending only on $n$ and the Lipschitz character of $\Omega$, and a compact set $K \subset \Omega$, such that
      \begin{enumerate}[label=(\alph*)]
        \item $\dist(K, \partial \Omega) = \min_{X \in K} \delta(X) \geq C_2$.
        \item If $u$ is harmonic in $\Omega$ and $\int_{\Omega} \delta_{\Omega}(X) |\nabla u(X)|^2 \, \D X < \infty$, then $u$ has non-tangential limit $f(Q)$ for almost every $Q \in \partial \Omega$ with respect to surface measure $\D \sigma$.
          Moreover, $f \in L^2(\partial \Omega, \D \sigma)$ and 
          \begin{align*}
            \int\limits_{\partial \Omega} f^2 \, \D \sigma \
            &\leq \ C_1 \left[\int_{\Omega} \delta(X) |\nabla u(X)|^2 \, \D X + \int_{K} u(X)^2 \, \D X \right].
          \end{align*}
      \end{enumerate}
    \end{thm}

    This theorem is due to Dahlberg \cite{D80}.  
    His proof relies on a real variable technique known as ``good-lambda'' inequalities.  We present here a more direct proof\footnote{\cite{DKPV97}      addresses systems of elliptic operators. Here, we specialize to the case of real-valued harmonic functions.} 
    from \cite{DKPV97}      that avoids using good lambda inequalities.    
    The main simplification over the proof in \cite{D80} comes from the use of an adapted distance function $\rho$, which has the same order of magnitude as the ordinary distance to the boundary, but also has further Carleson-type bounds on its second derivatives.

    \begin{lemma}[Adapted distance $\rho$]
      \label{lemma:4.2}
      Suppose that
      \begin{align*}
        \Omega \
        &= \ \left\{(x, y) \in \RR^n \times \RR : x \in \RR^n, y > \varphi(x) \right\}
      \end{align*}
      with $\varphi$ as above.
      There is a constant $C$ depending only on $L$ and $n$ and a Lipschitz continuous function $\rho$ on $\overline{\Omega}$ such that $\rho \in C^2(\Omega)$ and
      \begin{enumerate}[label=(\alph*)]
        \item $\rho(x, \varphi(x)) = 0$, $1/C \leq \p_y \rho(X) \leq C$, so that, in particular,
          \begin{align*}
            \frac{y - \varphi(x)}{C} \
            &\leq \ \rho(X) \
            \leq \ C \left(y - \varphi(x) \right) \text{ for all } X = (x, y) \in \Omega.
          \end{align*}
        \item $|\nabla \rho(X)| \leq C$ for all $X \in \Omega$.
        \item $\rho(X) |\nabla^2 \rho(X)|^2 \, \D X$ is a Carleson measure\footnote{A Carleson measure $\mu$ on $\Omega$ is defined as a measure such that for the bi-Lipschitz mapping $\Phi \colon \RR_{+}^{n + 1} \to \Omega$ given by $\Phi(x, y) = (x, y + \varphi(x))$, the pullback measure $\tilde{\mu}(E) = \mu(\Phi(E))$ is a Carleson measure on $\RR_{+}^{n + 1}$.} on $\Omega$ with norm bounded by $C$.
      \end{enumerate}
    \end{lemma}
    
    This distance function was first constructed by Dahlberg \cite{D86}, but the construction given in the published paper is a simpler one, due to Kenig and Stein.
    We present that construction here.
    Let
    \begin{align*}
      F(x, t) \
      &= \ C_n L t + \eta_t \ast \varphi(x), \quad x \in \RR^n,
    \end{align*}
    for $C_n$ a sufficiently large dimensional constant and $\eta_t(x) = t^{-n} \eta(x/t)$ is a smooth, compactly supported approximate identity on $\RR^n$.
    Then the function defined implicitly by
    \begin{align*}
      \rho(x, F(x, t)) \
      &= \ t, \quad t > 0, \quad \rho(x, \varphi(x)) = 0,
    \end{align*}
    satisfies all of the properties of Lemma \ref{lemma:4.2} (see \cite[pp. 1432-1433]{DKPV97} for further details).

    Define a cylindrical section of $\Omega$ and its bottom boundary by
    \begin{align*}
      \Omega_t \
      &= \ \left\{X = (x, y) \in \RR^n \times \RR : |x| < t, 0 < \rho(X) < t \right\}, \\
      \cB_t \
      &= \ \left\{(x, \varphi(x)) \in \RR^n \times \RR : |x| \leq t \right\}.
    \end{align*}

    \begin{defn}
      \label{defn:4.3}
      For a Lipschitz domain $\Omega$ and a function $f$ defined on $\Omega$, the non-tangential maximal function of $f$, denoted $N(f)$, is the function on $\partial \Omega$ given by
      \begin{align*}
        N(f)(Q) \
        &= \ \max_{X \in \Gamma(Q)} |f(X)|, \quad \Gamma(Q) \
        = \ \left\{X \in \Omega : |X - Q| \leq C \delta(X) \right\}.
      \end{align*}
    \end{defn}

    The constant $C$ is chosen sufficiently large depending on the dimension and the Lipschitz character of $\Omega$ so that $\Gamma(Q)$ plays the same role as a non-tangential cone when $\Omega = \RR_{+}^{n + 1}$.
    In other words,
    \begin{align*}
      \vol\left(\Gamma(Q) \cap B_r(Q) \right) \
      &\geq \ \frac{r^{n + 1}}{C}.
    \end{align*}
    Finally, define
    \begin{align*}
      A_{\Omega}(u) \
      &= \ \int\limits_{\Omega} \delta(X) \left|\nabla u \right|^2 \, \D X.
    \end{align*}

    \begin{lemma}[Main Lemma]
      \label{lemma:4.4}
      Suppose that $\Delta u = 0$ in $\Omega_3$, and that $u$, $\nabla u$, and $\nabla^2 u$ belong to $C^{\gamma}(\overline{\Omega}_3)$ for some $\gamma > 0$.
      Then there exists a constant $C$ depending only on $n$ and $L$ such that
      \begin{align*}
        \int\limits_{\cB_1} u^2 \, \D \sigma(Q) \
        &\leq \ C \left[A_{\Omega_3}(u) + A_{\Omega_3}(u)^{1/2} \left(\int_{\cB_2} N(u 1_{\Omega_2})^2 \, \D \sigma \right)^{1/2} + \int_{K} u^2 \, \D X\right]
      \end{align*}
      with
      \begin{align*}
        K \
        &= \ \left\{(x, y) \in \RR^n \times \RR : |x| \leq 3, 1 \leq \rho(X) \leq 3 \right\}.
      \end{align*}
      \begin{proof}
        Let $\cS_t$ be the level sets of $\rho$ in $\overline{\Omega}_2$, namely
        \begin{align*}
          \cS_t \
          &= \ \left\{X = (x, y) \in \RR^n \times \RR : |x| \leq 2, \rho(X) = t \right\}.
        \end{align*}
        Denote by $\D \sigma_t$ the surface measure on $\cS_t$.
        Evidently, $\cB_2 = \cS_0$ and $\D \sigma = \D \sigma_0$.

        Define $\theta \in C_0^{\infty}(\{x \in \RR^n : |x| < 2\})$ so that $\theta(x) = 1$ for all $|x| \leq 1$.
        Denote the upward-pointing unit normal to $\cS_t$ by
        \begin{align*}
          \nu(X) \
          &= \ \frac{\nabla \rho(X)}{|\nabla \rho(X)|} \
          = \ \left(\nu_1, \ldots, \nu_{n + 1} \right) \
          \in \ \RR^{n + 1}.
        \end{align*}
        The non-tangential limit of $\nu(X)$ as $X \to \partial \Omega$ exists pointwise almost everywhere $\D \sigma$ and equals the inner unit normal of $\partial \Omega$ relative to $\Omega$.
        Next, integrating by parts along vertical lines, we have
        \begin{align*}
          \int\limits_{\cB_1} u^2 \, \D \sigma \
          &\leq \ \int\limits_{\cB_2} \theta^2 u^2 \, \D \sigma \
          = \ - \int\limits_{\cS_2} \theta^2 u^2 \nu_{n + 1} \, \D \sigma_2 + \int\limits_{\{0 < \rho < 2\}} \p_y\left(\theta^2 u^2 \right) \, \D X.
        \end{align*}
        Using $\p_y \theta = 0$ and integration by parts on vertical lines, we can rewrite the second term as
        \begin{align*}
          &\int\limits_{\{0 < \rho < 2\}} \p_y \left(\theta^2 u^2 \right) \, \D X \\
          &\qquad \qquad = \ 2 \int\limits_{\{0 < \rho < 2\}} \theta^2 u \p_y u\left(\frac{\p_y \rho}{\p_y \rho} \right) \, \D X \\
          &\qquad \qquad = \ 2 \int\limits_{\{0 < \rho < 2\}} \p_y \left(\frac{\rho \theta^2 u \p_y u}{\p_y \rho} \right) \, \D X - 2 \int\limits_{\{0 < \rho < 2\}} \rho \theta^2 \p_y \left(\frac{u \p_y u}{\p_y \rho} \right) \, \D X \\
          &\qquad \qquad = \ 2 \int\limits_{\cS_2} \frac{\rho \theta^2 u \p_y u}{\p_y \rho} \, \D \sigma_2 - 2 \int\limits_{\{0 < \rho < 2\}} \rho \theta^2 \p_y \left(\frac{u \p_y u}{\p_y \rho} \right) \, \D X.
        \end{align*}
        The first of these terms in the last line satisfies
        \begin{align*}
          2 \left|\int_{\cS_2} \frac{\rho \theta^2 u \p_y u}{\p_y \rho} \, \D \sigma_2 \right| \
          &\leq \ C \max_{\cS_2} \left|u \p_y u \right| \sigma_2(\cS_2) \
          \leq \ C \int\limits_{K} u^2 \, \D X.
        \end{align*}
        Here, and from now on in this proof, the constant $C$ will denote a constant depending only on $n$ and $L$, which may change from one line to the next.
        The last inequality follows from the regularity of the harmonic function $u$.

        To bound the second term, we rewrite the integrand as
        \begin{align*}
          \rho \theta^2 \p_y \left(\frac{u \p_y u}{\p_y \rho} \right) \
          &= \ I + II + III
        \end{align*}
        with
        \begin{align*}
          I \
          &= \ \rho \theta^2 \frac{(\p_y u)^2}{\p_y \rho}, \quad II \
          = \ \rho \theta^2 \frac{u \p_y^2 u}{\p_y u}, \quad III \
          = \ -\rho \theta^2 \frac{(u \p_y u) \p_y^2 \rho}{(\p_y \rho)^2}.
        \end{align*}
        The first term is bounded by 
        \begin{align*}
          \left|\int_{\{0 < \rho < 2\}} I \, \D X \right| \
          &\leq \ C \int\limits_{\{0 < \rho < 2\}} \theta^2 \rho |\nabla u|^2 \, \D X \
          \leq \ C A_{\Omega_2}.
        \end{align*}
        The third term is bounded using the Cauchy-Schwarz inequality, the Carleson measure bound (Lemma \ref{lemma:4.2} (c)), and the Carleson theorem (Theorem \ref{thm:1.5}), which yield
        \begin{align*}
          \left|\int_{\{0 < \rho < 2\}} III \, \D X \right| \
          &\leq \ C \int\limits_{\{0 < \rho < 2\}} \theta^2 |u \nabla u| \rho |\p_y^2 \rho| \, \D X \\
          &\leq \ C \left(\int_{\Omega_2} \rho |\nabla u|^2 \, \D X \right)^{1/2} \left(\int_{\Omega_2} u^2 \rho^2 |\p_y^2 \rho| \, \D X \right)^{1/2} \\
          &\leq \ C A_{\Omega_3}^{1/2} \left(\int_{\cB_2} N(u 1_{\Omega_2})^2 \, \D \sigma \right)^{1/2}.
        \end{align*}
        It remains to bound the term with integrand II.
        For that purpose, define vector fields 
        \begin{align*}
          T_j \
          &= \ \nu_{n + 1} \p_j - \nu_j \p_y, \quad \p_j \
          = \ \frac{\p}{\p x_j}, \quad j \
          = \ 1, 2, \ldots, n
        \end{align*}
        tangent to $\cS_t$.

        We claim that for any $C^1$ functions $f_1$ and $f_2$ on $\Omega_3$,
        \begin{align}
          \int\limits_{\cS_t} \theta^2(T_j f_1) f_2 \, \D \sigma_t \
          &= \ - \int\limits_{\cS_t} f_1\left(T_j(\theta^2 f_2) \right) \, \D \sigma_t, \quad 0 < t \leq 2.
          \label{eq:17typed}
        \end{align}
        Indeed, fix $t$.
        Since $\rho(x, F(x, t)) = t$, the change of variables $x = z$, $y = F(z, t)$, $z \in \RR^n$ gives
        \begin{align*}
          \D \sigma_t \
          &= \ \sqrt{1 + |\nabla_z F(z, t)|^2} \, \D z, \quad \nu \
          = \ \frac{(- \nabla_z F, 1)}{\sqrt{1 + |\nabla_z F|^2}},
        \end{align*}
        and
        \begin{align*}
          \frac{\p f}{\p z_j} \
          &= \ \p_j f + \frac{\p F}{\p z_j} \p_y f \\
          &= \ \sqrt{1 + |\nabla_z F(z, t)|^2} \left(\nu_{n + 1} \p_j - \nu_j \p_y \right) f \\
          &= \ \sqrt{1 + |\nabla_z F(z, t)|^2} T_j f.
        \end{align*}
        Therefore,
        \begin{align*}
          \int\limits_{\cS_t} \theta^2\left(T_j f_1 \right) f_2 \, \D \sigma_t \
          &= \ \int\limits_{|z| < 2} \theta(z)^2 \left(\frac{\p}{\p z_j} f_1\left(z, F(z, t) \right)\right) f_2\left(z, F(z, t) \right) \, \D z \\
          &= \ - \int\limits_{|z| < 2} f_1\left(z, F(z, t) \right) \frac{\p}{\p z_j} \left(\theta(z)^2 f_2(z, F(z, t)) \right) \, \D z \\
          &= \ - \int\limits_{\cS_t} f_1 T_j\left(\theta^2 f_2 \right) \, \D \sigma_t.
        \end{align*}
        Next, since $\nu_{n + 1}^2 + \sum_{j=1}^{n} \nu_j^2 = 1$ and $\Delta u = 0$, we have
        \begin{align}
          \p_y^2 u \
          &= \ \sum_{j=1}^{n} \nu_j T_j \p_y u + \nu_{n + 1} \sum_{j=1}^{n} T_j \p_j u.
          \label{eq:18typed}
        \end{align}
        Finally, using the co-area formula \cite{EG2025}, \eqref{eq:18typed}, \eqref{eq:17typed}, and the property that $T_j \rho = 0$, we find that
        \begin{align*}
          &\left|\int_{\{0 < \rho < 2\}} II \, \D X \right| \\
          &\qquad = \ \left|\sum_j \int_{0}^{2} \int_{\cS_t} \rho(\p_y u) T_j \left(\frac{\theta^2 u \p_j \rho}{|\nabla \rho|^2 \p_y \rho} \right) + \rho(\p_j u) T_j \left(\frac{\theta^2 u}{|\nabla \rho|^2} \right) \, \D \sigma_t \, \D t \right| \\
          &\qquad \leq \ C \int\limits_{\Omega_2} \rho |\nabla u|^2 \, \D X + C \int\limits_{\Omega_2} |u| \, |\nabla u| (\rho |\nabla^2 \rho|) \, \D X + C \int\limits_{\Omega} \rho |u| \, |\nabla u| \, |\nabla \theta^2| \, \D X \\
          &\leq \ C A_{\Omega_3}(u) + C A_{\Omega_3}(u)^{1/2} \left(\int_{\cB_2} N(u 1_{\Omega_2}) \, \D \sigma \right)^{1/2}.
        \end{align*}
        The first inequality above follows by grouping terms of three types.
        The first type is the ones in which $T_j$ acts on $u$.
        The second type is the ones for which $T_j$ acts on $\p_j \rho$, $\p_y \rho$, and $|\nabla \rho|^2$.
        The last type arises when $T_j$ acts on $\theta^2$.
        For the second inequality, the first term is dominated by $C A_{\Omega_3}$ simply because $\Omega_2 \subset \Omega_3$.
        The second term is controlled in exactly the same way as term $III$ above.
        The final term is estimated using only elementary inequalities.
        The Cauchy-Schwarz inequality implies
        \begin{align*}
          \int\limits_{\Omega_2} \rho |u| \, |\nabla u| \, |\nabla \theta^2| \, \D X \
          &\leq \ \left(\int_{\Omega_2} \rho |\nabla u|^2 \, \D X \right)^{1/2} \left(\int_{\Omega_2} \rho u^2 \, \D X \right)^{1/2}.
        \end{align*}
        Moreover, since $|u(z, F(z, t))| \leq N(u)(z, \varphi(z))$ and $\sqrt{1 + |\nabla_z F(z, t)|^2} \leq C$,
        \begin{align*}
          \int\limits_{\Omega_2} \rho u^2 \, \D X \
          &= \ \int\limits_{0}^{2} t \int\limits_{\cS_t} u^2 \frac{\D \sigma_t}{|\nabla \rho|} \, \D t \\
          &\leq \ C \int\limits_{0}^{2} t \, \D t \int\limits_{\cB_2} N(u 1_{\Omega_2})^2 \, \D \sigma \\
          &= \ C \int\limits_{\cB_2} N(u 1_{\Omega_2})^2 \, \D \sigma.
        \end{align*}
        Thus, the third term is dominated by the same expression as the second term.
        This concludes the proof of Lemma \ref{lemma:4.4}.
      \end{proof}
    \end{lemma}

    Recall that on a Lipschitz domain $\Omega$, we define
    \begin{align*}
      \delta(X) \
      &= \ \min_{Q \in \partial \Omega} |X - Q|, \quad \Gamma(Q) \
      = \ \left\{X \in \Omega : |X - Q| \leq C \delta(X) \right\}
    \end{align*}
    for $C$ a sufficiently large constant depending only on the Lipschitz character of $\Omega$.

    \begin{lemma}
      \label{lemma:4.5}
      Let $\Omega$ be a bounded Lipschitz domain in $\RR^{n + 1}$, $n \geq 1$.
      There exist constants $C_1, C_2 > 0$ depending only on $n$ and the Lipschitz character of $\Omega$, and a compact set $K \subset \Omega$, such that
      \begin{enumerate}[label=(\alph*)]
        \item $\dist(X, \partial \Omega) \geq C_2$ for all $X \in K$.
        \item If $u$ is harmonic in $\Omega$ and $u$, $\nabla u$, and $\nabla^2 u$ belong to $C^{\gamma}(\overline{\Omega})$ for some $\gamma > 0$, then
          \begin{align*}
            \int\limits_{\partial \Omega} u^2 \, \D \sigma \
            &\leq \ C_1 \left[A_{\Omega}(u) + A_{\Omega}(u)^{1/2} \left(\int_{\partial \Omega} N_{\Omega}(u)^2 \, \D \sigma \right)^{1/2} + \int_{K} u^2 \, \D X \right].
          \end{align*}
      \end{enumerate}
    \end{lemma}

    Lemma \ref{lemma:4.5} follows from Lemma \ref{lemma:4.4} by an easy covering argument.

    \begin{lemma}
      \label{lemma:4.6}
      Under the assumptions of Lemma \ref{lemma:4.5}, there exist constants $C_3, C_4 > 0$ depending only on $n$ and the Lipschitz character of $\Omega$, and a compact set $K \subset \Omega$, such that $\dist(X, \partial \Omega) \geq C_3$ for all $X \in K$, and
      \begin{align*}
        \int\limits_{\partial \Omega} N(u)^2 \, \D \sigma \
        &\leq \ C_4 \left[A_{\Omega}(u) + \int_{K} u^2 \, \D X \right].
      \end{align*}
      \begin{proof}
        By work of Dahlberg \cite{D79}, every harmonic function $u$ on a bounded Lipschitz domain that is continuous up to the boundary satisfies
        \begin{align*}
          \int\limits_{\partial \Omega} N_{\Omega}(u)^2 \, \D \sigma \
          &\leq \ C \int\limits_{\partial \Omega} u^2 \, \D \sigma
        \end{align*}
        with $C$ a constant depending only on the dimension and the Lipschitz character of $\Omega$.
        This combined with Lemma \ref{lemma:4.5} implies
        \begin{align*}
          \int\limits_{\partial \Omega} N_{\Omega}(u)^2 \, \D \sigma \
          &\leq \ C C_1 \left[A_{\Omega}(u) + A_{\Omega}(u)^{1/2} \left(\int_{\partial \Omega} N(u)^2 \, \D \sigma \right)^{1/2} + \int_K u^2 \, \D X \right] \\
          &\leq \ C C_1 \left[A_{\Omega}(u) + \frac{1}{\epsilon} A_{\Omega}(u) + \epsilon \int_{\partial \Omega} N_{\Omega}(u)^2 \, \D \sigma + \int_K u^2 \, \D X \right].
        \end{align*}
        Since $\int_{\partial \Omega} N_{\Omega}(u)^2 \, \D \sigma < \infty$, we can choose $\epsilon = 2/C C_1$, and subtract, to show that
        \begin{align*}
          \int\limits_{\partial \Omega} N_{\Omega}(u)^2 \, \D \sigma \
          &\leq \ 2 C C_1 \left[(1 + 1/\epsilon) A_{\Omega}(u) + \int_{K} u^2 \, \D X \right],
        \end{align*}
        which proves Lemma \ref{lemma:4.6}.
      \end{proof}
    \end{lemma}

    \begin{proof}[Conclusion of the proof of Theorem \ref{thm:4.1}]

      We now deduce our theorem using approximation.
      Following Ne\u cas \cite{N2012} and Verchota \cite{V84}, there exist $C^{\infty}$ domains $U_j$ such that $\overline{U}_j \subset \Omega$ and $U_j$ have uniformly bounded Lipschitz character.
      Furthermore, there are Lipschitz homeomorphisms $\Lambda_j \colon \partial \Omega \to \partial U_j$ with $|Q - \Lambda_j(Q)| \to 0$ uniformly for $Q \in \partial \Omega$ as $j \to \infty$.

      Consider a harmonic function $u$ on $\Omega$ such that $A_{\Omega}(u) < \infty$.
      Since $u \in C^{\infty}(\overline{U}_j)$ for all $j$, we can apply Lemma \ref{lemma:4.6} to obtain
      \begin{align*}
        \int\limits_{\partial U_j} N_{U_j}(u)^2 \, \D \sigma_j \
        &\leq \ C_3 \left(A_{U_j}(u) + \int_K u^2 \, \D x\right) \
        \leq \ C_3 \left(A_{\Omega}(u) + \int_K u^2 \, \D x \right)
      \end{align*}
      for some compact $K \subset \Omega$ bounded away from $\partial \Omega$.
      For each fixed $\epsilon > 0$, the truncated maximal function
      \begin{align*}
        N_{\epsilon}(u)(Q) \
        &\coloneqq \ \sup\left\{|u(X)| : X \in \Gamma_{\Omega}(Q), |X - Q| > \epsilon \right\}
      \end{align*}
      satisfies 
      \begin{align*}
        N_{\epsilon}(u)(Q) \
        &\leq \ N_{U_j}(u)(\Lambda_j(Q))
      \end{align*}
      for sufficiently large $j$ depending on $\epsilon$.
      Since $N_{\epsilon}(u)(Q)$ is increasing in $\epsilon$ and tends to $N_{\Omega}(Q)$ for each $Q$, the monotone convergence theorem implies
      \begin{align*}
        \int\limits_{\partial \Omega} N_{\Omega}(u)^2 \, \D \sigma \
        &\leq \ C_3 \left(A_{\Omega}(u) + \int_{K} u^2 \, \D x \right).
      \end{align*}
      A theorem of Hunt and Wheeden \cite{HW68} says that the non-tangential limit
      \begin{align*}
        \lim_{X \in \Gamma_{\Omega}(Q)} u(X) \
        &= \ f(Q)
      \end{align*}
      exists almost everywhere with respect to harmonic measure on $\partial \Omega$.
      A theorem of Dahlberg \cite{D77} says that harmonic measure and surface measure $\D \sigma$ are mutually absolutely continuous.
      Thus, by the dominated convergence theorem, $|f(Q)| \leq N_{\Omega}(u)(Q)$ almost everywhere $\D \sigma$, $f \in L^2(\partial \Omega, \D \sigma)$, and 
      \begin{align*}
        \int\limits_{\partial \Omega} f^2 \, \D \sigma \
        &\leq \ C_3 \left(\int_{\Omega} \delta_{\Omega}(X) |\nabla u(X)|^2 \, \D X + \int_{K} u^2 \, \D x \right)
      \end{align*}
      as desired.
    \end{proof}

    \begin{cor}
      \label{cor:4.7}
      Let $\Omega$ be a bounded Lipschitz domain in $\RR^{n + 1}$, $n \geq 1$.
      Assume that $u$ is harmonic in $\Omega$ and $u \in L_{3/2}^2(\Omega)$.
      Denote by $f$ the non-tangential pointwise limit of $u$, almost everywhere $\D \sigma$.
      Then $f \in L_1^2(\partial \Omega)$ and $f = \Tr(u)$, where $\Tr(u)$ is the Sobolev trace of $u$ on $\partial \Omega$.
      \begin{proof}
        Let $W_j = \frac{\p u}{\p x_j}$ for $j = 1, 2, \ldots, n + 1$.
        Then $W_j \in L_{1/2}^2(\Omega)$, and by \cite[Theorem 4.2]{JK95}, the area integral square norm is equivalent modulo constants to the square of the $L_{1/2}^2(\Omega)$ norm, and consequently,
        \begin{align*}
          \int\limits_{\Omega} \delta_{\Omega}(X) \left|\nabla W_j \right|^2 \, \D X \
          &< \ \infty.
        \end{align*}
        Thus, as in the proof of Theorem \ref{thm:3.1},
        \begin{align*}
          \int\limits_{\partial \Omega} N_{\Omega}(W_j)^2 \, \D \sigma \
          &< \ \infty.
        \end{align*}
        Moreover, for almost every $Q \in \partial \Omega$ with respect to $\D \sigma$, $\nabla u$ has a non-tangential limit $\vec{W}(Q) = (W_1, W_2, \ldots, W_{n + 1})$ with $W_j \in L^2(\partial \Omega, \D \sigma)$.

        To show that $\Tr(u) \in L_1^2(\partial \Omega)$, it will suffice to prove it for $u$ defined in a coordinate chart.
        \begin{lemma}
          \label{lemma:4.8}
          Let $U_1 \subset \RR^{n + 1}$ be the unit cylinder and $\cD_1 \subset \RR^n$ the unit disk, defined by
          \begin{align*}
            U_1 \
            &= \ \left\{(z, t) \in \RR^n \times \RR : |z| < a, 0 < t < 1 \right\}; \quad \cD_1 \
            = \ \left\{z \in \RR^n : |z| < 1 \right\}.
          \end{align*}
          Suppose $v$ is a locally Lipschitz continuous function on $U_1$ belonging to the Sobolev space $L_1^2(U_1)$.
          Suppose, in addition, that there exists an $\RR^n$-valued function $\vec{V} \in L^2(\cD_1, \D z)$ such that
          \begin{align*}
            \lim_{t \to 0^{+}} \int\limits_{\cD_1} \left|\nabla_z v(z, t) - \vec{V}(z) \right|^2 \, \D z \
            &= \ 0.
          \end{align*}
          Then
          \begin{enumerate}[label=(\alph*)]
            \item There exists $v_0 \in L^2(\cD_1, \D z)$ such that $v_0(z) = \lim_{t \to 0^{+}} v(z, t)$ for almost every $z \in \cD_1$ with respect to $\D z$.
            \item $\lim_{t \to 0^{+}} \int_{\cD_1} |v(z, t) - v_0(z)|^2 \, \D z = 0$.
            \item $v_0(z) = \Tr(v)(z, 0)$ for almost every $z \in \cD_1$, where $\Tr(v)$ is the Sobolev trace of $v$ on the bottom boundary $\{t = 0\} \cap \partial U_1$.
            \item $\nabla_z v_0(z) = \vec{V}(z)$.
              In particular, $v_0 \in L_1^2(\cD_1)$.
          \end{enumerate}
          \begin{proof}
            Since $v$ is absolutely continuous in $t$ for $t > 0$, we have
            \begin{align*}
              v_0(z) \
              &= \ -v(z, 1/2) - \int\limits_{0}^{1/2} \p_t v(z, t) \, \D t \text{ exists and } \ v_0(z) = \lim_{t \to 0^{+}} v(z, t),
            \end{align*}
            provided we can show that the integral is convergent.
            In fact, the bound 
            \[\int_{U_1} |\nabla v|^2 \, \D z \, \D t < \infty \] and Fubini's theorem imply that the integral is absolutely convergent for almost every $z$ with respect to $\D z$.
            Furthermore, for all $0 < t_1 < t_2 < 1$,
            \begin{align*}
              \left|v(x, t_2) - v(x, t_1) \right|^2 \
              &= \ \left|\int_{t_1}^{t_2} \p_t v(x, t) \, \D t \right|^2 \\
              &\leq \ \left(t_2 - t_1 \right) \int\limits_{t_1}^{t_2} \left(\p_t v(x, t) \right)^2 \, \D t \\
              &\leq \ \left(t_2 - t_1 \right) \int\limits_{0}^{1} \left(\p_t v(x, t) \right)^2 \, \D t.
            \end{align*}
            Integrating in $z$, we have
            \begin{align*}
              \int\limits_{\cD_1} \left|v(z, t_2) - v(z, t_1) \right|^2 \, \D z \
              &\leq \ \left(t_2 - t_1 \right) \int\limits_{U_1} \left|\p_t v(z, t) \right|^2 \, \D z \, \D t.
            \end{align*}
            Using a telescoping sum, it follows that
            \begin{align*}
              \lim_{j \to \infty} \int\limits_{\cD_1} \left|v(z, 2^{-j}) - v_0(z) \right|^2 \, \D z \
              &= \ 0,
            \end{align*}
            and from this, one easily deduces not only that $v_0 \in L^2(\cD_1, \D z)$, but also part (b).

            To prove (c), note that the subspace of Lipschitz continuous functions on $\overline{U}_1$ is dense in $L_1^2(U_1)$.
            The reasoning above shows that for $v$ in this subspace,
            \begin{align*}
              \int\limits_{\cD_1} \left|v(z, s) - v(z, 0) \right|^2 \, \D z \
              &\leq \ s \int\limits_{U_1} \left|\nabla v(z, t) \right|^2 \, \D z \, \D t.
            \end{align*}
            The ordinary restriction of such continuous functions coincides with the Sobolev trace on every set $\{t = s\}$.
            Then by continuity of the Sobolev trace operation from $L_1^2(U_1)$ to $L^2(\{t = s\})$, we can pass from the dense class to all of $L_1^2(U_1)$, establishing (c).

            Finally, we prove part (d) by a limiting argument.
            Since as $t \to 0^{+}$,
            \begin{align*}
              v(\cdot, t) \to v_0(\cdot) \text{ and } \nabla_z v(\cdot, t) \to \vec{V}(\cdot) \text{ in } L^2(\D z),
            \end{align*}
            the distribution $\nabla_z v_0$ is defined for test functions $g \in C_0^{\infty}(\cD_1)$ by
            \begin{align*}
              \left(\nabla_z v_0 \right)(g) \
              &= \ -\int\limits_{\cD_1} v_0(z) \nabla_z g(z) \, \D z \
              = \ -\lim_{t \to 0^{+}} \int\limits_{\cD_1} v(z, t) \nabla_z g(z) \, \D z \\
              &= \ \lim_{t \to 0^{+}} \nabla_z v(z, t) g(z) \, \D z \
              = \ \int\limits_{\cD_1} \vec{V}_j(z) g(z) \, \D z
            \end{align*}
            as desired.
          \end{proof}
        \end{lemma}
        We can now deduce Corollary \ref{cor:4.7}. 
        For any Lipschitz function $\varphi$, consider the mapping
        \begin{align*}
          \Psi \colon U_1 \to \RR^{n + 1} \text{ given by } \Psi(z, t) \
          &= \ (z, t + \varphi(z))
        \end{align*}
        and define
        \begin{align*}
          \Omega_1 \
          &= \ \Psi\left(U_1 \right), \quad \cB_1 \
          = \ \left\{\Psi(z, 0) : z \in \cD_1 \right\}.
        \end{align*}
        Every coordinate chart of $\Omega$ in the neighborhood of $\partial \Omega$ can be written in this form up to rigid motion and dilation.
        If $\Omega_1$ is a coordinate chart of $\Omega$ with $\cB_1 \subset \partial \Omega$, and $u$ satisfies the hypotheses of Corollary \ref{cor:4.7}, then $v(z,t) = u(\Psi(z, t))$ satisfies the hypotheses of Lemma \ref{lemma:4.8}.
        Indeed,
        \begin{align*}
          \nabla_z v(z, t) \
          &= \ \nabla_x u(x, y) + \nabla_z \varphi(z) \p_y u(x, y) \text{ for } (x, y) = \Psi(z, t), \quad (z, t) \in U_1,
        \end{align*}
        so that $\nabla v \in L_{loc}^{\infty}(U_1)$ and $v \in H^1(U_1)$.
        Recall that $\vec{W} = \nabla_{xy} u$.
        The estimate above for $N_{\Omega}(W_j)$, $1 \leq j \leq n + 1$, and the dominated convergence theorem imply
        \begin{align*}
          \lim_{t \to 0} \int\limits_{\cD_1} \left|\frac{\p v}{\p z_j}(z, t) - \left[W_j(\Psi(z, 0)) + \frac{\p \varphi}{\p z_j} W_{n + 1}(\Psi(z, 0)) \right]\right|^2 \sqrt{1 + |\nabla_z \varphi(z)|^2}
          = 0.
        \end{align*}
        Thus, the final hypothesis of Lemma \ref{lemma:4.8} is satisfied with 
        \begin{align*}
          V_j(z) \
          &= \ W_j(\Psi(z, 0)) + \frac{\p \varphi}{\p z_j}(z) W_{n + 1}(\Psi(z, 0)), \quad j = 1, \ldots, n.
        \end{align*}
        Applying the lemma, we have, in particular, $v_0 \in L_1^2(\cD_1)$.
        Finally, the bi-Lipschitz change of variables $\Psi(\cdot, 0)$ preserves non-tangential limits and sends $L_1^2(\D z)$ to $L_1^2(\cB_1)$.
        This concludes the proof of Corollary \ref{cor:4.7}.
      \end{proof}
    \end{cor}

    For completeness, we show that the Sobolev trace and the non-tangential limit also coincide at lower regularity, all the way down to $L_{\alpha}^2$ for $\alpha > 1/2$.
    At the endpoint $\alpha = 1/2$, the Sobolev trace is no longer defined.

    \begin{prop}
      \label{prop:4.9}
      If $u \in L_{\alpha}^2(\Omega)$ for some $\alpha > 1/2$, and $u$ is harmonic, then $u$ has non-tangential limits $f \in L^2(\partial \Omega, \D \sigma)$ and $f$ coincides with the Sobolev space trace
      \begin{align*}
        f \
        &= \ \Tr(u).
      \end{align*}
      \begin{proof}
        Without loss of generality, assume in addition that $\alpha \leq 1$ (higher regularity will play no role in the proof).
        As in Lemma \ref{lemma:4.4}, it will suffice to treat one unit-sized coordinate chart.
        Consider a Lipschitz function $\varphi \colon \RR^n \to \RR$ and define
        \begin{align*}
          U \
          &= \ \left\{(x, y) \in \RR^{n} \times \RR : \varphi(x) < y < \varphi(x) + 3, \ |x| < 3 \right\}, \\
          \cB_1 \
          &= \ \left\{(x, y) \in \RR^{n} \times \RR : |x| < 1, y = \varphi(x) \right\}.
        \end{align*}
        Then consider a cut-off function $\theta \in C_0^{\infty}(U)$ such that $\theta = 1$ on $|x| < 2$, $|y - \varphi(x)| < 2$.
        Define
        \begin{align*}
          v(x, t) \
          &= \ \theta^2(x, t) u(x, \varphi(x) + t), \quad t > 0.
        \end{align*}
        The bi-Lipschitz change of variables preserves the Sobolev spaces $H^{\alpha}$ for $0 < \alpha \leq 1$.

        \begin{lemma}
          \label{lemma:4.10}
          Suppose that $1/2 < \gamma < 3/2$ and $v \in L_{\gamma}^2(\RR^{n + 1})$.
          Then
          \begin{align*}
            \int\limits_{\RR^n} \left|v(x, t) - v(x, 0) \right|^2 \, \D x \
            &\leq \ \int\limits_{\RR} \int\limits_{\RR^n} \left|\hat{v}(\xi, \tau) \right|^2 |(\xi, \tau)|^{2 \gamma} \, \D \xi \, \D \tau \\
            &= \ C y^{2 \gamma - 1} \|v\|_{H^{\gamma}(\RR^{n + 1})}^2.
          \end{align*}
          \begin{proof}
            By Plancherel's theorem and the Cauchy-Schwarz inequality,
            \begin{align*}
              &\int_{\RR^n} \left|v(x, t) - v(x, 0) \right|^2 \, \D x \\
              &= \ \int_{\RR^n} \left|\int_{\RR} (e^{\i t r} - 1) \hat{v}(\xi, \tau) \, \D \tau \right|^2 \, \D \xi \\
              &\leq \ \int_{\RR^n} \left(\int_{\RR} |e^{\i t \tau} - 1|^2(|\xi| + |\tau|)^{-2 \gamma} \, \D \tau \right) \left(\int_{\RR} |\hat{v}(\xi, \tau)|^2 (|\xi| + |\tau|)^{2 \gamma} \, \D \tau \right) \, \D \xi \\
              &\leq \ C_{\gamma} t^{2 \gamma - 1} \int_{\RR^n} \int_{\RR} |\hat{u}(\xi, \tau)|^2 (|\xi| + |\tau|)^{2 \gamma} \, \D \tau \, \D \xi.
            \end{align*}
            The second inequality follows from the computation
            \begin{align*}
              \int\limits_{\RR} \left|e^{\i t \tau} - 1 \right|^2 (|\xi| + |\tau|)^{-2 \gamma} \, \D \tau \
              &\leq \ \int\limits_{\RR} |e^{\i t \tau} - 1|^2 |\tau|^{-2 \gamma} \, \D \tau \\
              &= \ t^{2 \gamma - 1} \int\limits_{\RR} \left|e^{\i k} - 1 \right|^2 |k^{-2 \gamma} \, \D k \\
              &= \ C_{\gamma} t^{2 \gamma - 1}.
            \end{align*}
            The integral over $k \in \RR$ is finite because $1/2 < \gamma < 3/2$.
          \end{proof}
        \end{lemma}

        The lemma says, in particular, that the Sobolev trace $v(\cdot, t)$ is continuous in the variable $t$ in the $L^2(\RR^n)$ norm. (In fact, in our case, $v$ is continuous for $t > 0$, so the trace is defined pointwise everywhere in the strict upper half plane and equals its Sobolev trace. The purpose of the lemma is to show this continuous function converges in $L^2(\RR^n)$ norm to the trace boundary values.)

        Since $u \in L_{\alpha}^2(\Omega) \subset L_{1/2}^2(\Omega)$, the same argument as in Corollary \ref{cor:4.7}, applied to $u$ instead of $w$, shows that $u$ has non-tangential limits $f$.
        By the dominated convergence theorem, and the fact that $\D \sigma$ is comparable to $\D x$ on $y = \varphi(x)$, we have
        \begin{align*}
          \lim_{t \to 0^{+}} \int\limits_{\{|x| < 1\}} \left|v(x, t) - f(x) \right|^2 \, \D x
          &= \lim_{t \to 0^{+}} \int\limits_{\{|x| < 1\}} \left|u(x, \varphi(x) + t) - f(x) \right|^2 \, \D x
          = 0.
        \end{align*}
        Thus, $f(x) = v(x, 0)$ is the Sobolev trace.
      \end{proof}
    \end{prop}

  \appendix
  \section{Appendix}
    As we mentioned in Section 2, we deferred until now a second proof of Theorem \ref{thm:2.1} for graphs in the plane above a Lipschitz function with arbitrary 
    Lipschitz constant $L$.
    The proof consists of using the conformal mapping of $\RR_{+}^2$ to $\Omega$ to transform the estimate of Theorem 2.1 into a special case of a weighted inequality due to Gundy and Wheeden \cite{GW74}.
    We will use notation as in Section 2.
    Recall that non-tangential cones in $\RR_{+}^2$ are denoted by
    \begin{align*}
      \Gamma((x, 0)) \
      &= \ \left\{(x', y') : |x' - x| < y \right\}.
    \end{align*}
    For $u$ defined in $\RR_{+}^2$, the Lusin area function is defined by
    \begin{align*}
      S(u)(x)^2 \
      &= \ \int\limits_{\Gamma((x, 0))} \left|\nabla u(x', y') \right|^2 \, \D x' \, \D y'.
    \end{align*}
    Recall that the non-tangential maximal function is defined by
    \begin{align*}
      N(u)(x) \
      &= \ \sup_{Y \in \Gamma((x, 0))} \left|u(Y) \right|.
    \end{align*}
    The special case of the Gundy-Wheeden theorem that we need is as follows.

    \begin{thm}
      \label{thm:A.1}
      Suppose that $w$ is an $A_{\infty}$ weight on $\RR$, and $u$ is a harmonic function in the upper half-plane $\RR_{+}^2$.
      If $S(u) \in L^2(w \, \D x)$, then there is a unique constant $c$ such that $u(x, y) - c$ tends to zero as $y \to \infty$, the non-tangential maximal function of $u - c$ belongs to $L^2(w \, \D x)$, and 
      \begin{align*}
        \int\limits_{\RR} N(u - c)^2 w(x) \, \D x \
        &\leq \ C \int\limits_{\RR} S(u)(x)^2 \, \D x,
      \end{align*}
      with a constant depending only on the $A_{\infty}$ constant of the weight $w$.
    \end{thm}
\begin{proof}  The weight $w$ to which we apply this theorem is $w(x) = |\Phi'(x)|$, which we showed in Theorem \ref{thm:1.9} (d) belongs to the Muckenhoupt class $A_2$.
    As is well-known, the class $A_2$ is a subset of the class $A_{\infty}$ (see \cite[p. 196, 1.7]{S93}).
    In fact, the class $A_{\infty}$ is sometimes defined as the union of the classes $A_p$ for $1 < p < \infty$ \cite[p. 196, 1.7]{S93}.

    Recall in the notation of Theorem \ref{thm:2.1}, $\zeta = \xi + \i \eta$ and 
    \begin{align*}
      A \
      &= \ \int\limits_{\Omega} \delta(\zeta) \left|F'(\zeta) \right|^2 \, \D \xi \, \D \eta.
    \end{align*}
    Let $\tilde{F}(z) = F \circ \Phi(z)$, $z \in \RR_{+}^2$.
    Changing variables and using Proposition \ref{prop:1.11}, 
    \begin{align*}
      A \
      &= \ \int\limits_{\RR_{+}^2} \delta(\Phi(z)) \left|\tilde{F}'(z) \right|^2 \, \D x \, \D x \
      \simeq \ \int\limits_{\RR_{+}^2} y \left|\Phi'(z) \right| \left|\tilde{F}'(z) \right|^2 \, \D x \, \D y, 
    \end{align*}
    with $z = x + \i y$.
    Next note that if $I = \{s \in \RR : |s - x| < y\}$, then the lower bound on the Poisson kernel $y/\pi(t^2 + y^2) \geq 1/2 \pi y$ for $|t| \leq y$ yields, for $w(s) = |\Phi'(s)|$,
    \begin{align*}
      \int\limits_{I} w(s) \, \D s \
      &\lesssim \ \int\limits_{I} \re (\Phi'(s)) \, \D s \
      \leq \ 2 \pi y \re (\Phi'(x + \i y)) \
      \lesssim \ y \left|\Phi'(x + \i y) \right|.
    \end{align*}
    Therefore,
    \begin{align*}
      \int_{\RR} S(\tilde{F})^2(s) w(s) \, \D s \
      &= \ \int_{\RR} \left(\int_{|s - x| < y} \left|\tilde{F}(x + \i y) \right|^2 \, \D x \, \D y \right) w(s) \, \D s \\
      &= \ \int_{\RR_{+}^2} \left|\tilde{F}'(x + \i y) \right|^2 \left(\int_{|s - x| < y} w(s) \, \D s \right) \, \D x \, \D y \\
      &\lesssim \ \int\limits_{\RR_{+}^2} \left|\tilde{F}'(x + \i y) \right|^2 y \left|\Phi'(x + \i y) \right| \, \D x \, \D y \\
      &\simeq \ A.
    \end{align*}
    Since $F(\zeta) \to 0$ as $|\zeta| \to \infty$, provided $\delta(\zeta) \geq \delta_0 > 0$, the same holds for $\tilde{F}(x + \i y)$ in the region $y \geq y_0 > 0$.
    Hence, Theorem \ref{thm:A.1} shows that
    \begin{align*}
      \int\limits_{\RR} N(\tilde{F})^2 w(x) \, \D x \
      &\lesssim \ A.
    \end{align*}
    In particular, $\tilde{F}$ is bounded on almost every cone $\Gamma((x, 0))$ with respect to $w(x) \, \D x$.
    By Calder\'on's theorem (see \cite[Theorem 3, p.201]{S70}), the non-tangential limit $\tilde{F}(x)$ exists almost everywhere and
    \begin{align*}
      \int\limits_{\RR} \left|\tilde{F}(x) \right|^2 w(x) \, \D x \
      &\lesssim \ A.
    \end{align*}
    By Remark \ref{rmk:1.12}, $F$ also has non-tangential limits at  almost every $\zeta \in \partial \Omega$ with respect to $\D \sigma$, and $F(\zeta) = \tilde{F}(x)$ for $x = \Phi^{-1}(\zeta)$.
    In all,
    \begin{align*}
      \int\limits_{\partial \Omega} \left|F(\zeta) \right|^2 \, \D \sigma(\zeta) \
      &= \ \int\limits_{\RR} \left|\tilde{F}(x) \right|^2 w(x) \, \D x \
      \lesssim \ A
    \end{align*}
    as desired.
    \end{proof}
    
    As explained  in Section 2, we also deferred to this appendix the one remaining step to prove Corollary \ref{cor:2.2written}.  Recall that it suffices to prove
    the following proposition.
        \begin{prop} 
       Let $\Omega$ and $\epsilon_0 > 0$ be as in Theorem \ref{thm:2.1}.
    Suppose that $F$ is analytic in $\Omega$, and 
      \[
      \int_{\Omega} \delta(\zeta) |F'(\zeta)|^2 \, \D \xi \, \D \eta = A < \infty, 
      \quad (\zeta = \xi + \i \eta).
      \]
      Then there exists $c\in\CC$ such that for any $\delta_0 > 0$,
      \begin{align*}
        F(\zeta) \to c  \ \text{ as } |\zeta| \to \infty  \quad {for} \ \ \zeta\in \Omega \ \ \mbox{satisfying} \ \  \delta(\zeta) \geq \delta_0.
      \end{align*}
      \label{prop:A.2}  
    \end{prop}
 \begin{proof} The main step of the argument is the same as the one above
 between \eqref{eq:14typed}  and \eqref{eq:15typed}, localized to cones $\Gamma_\alpha(\zeta)$ defined below.     For each $\zeta\in \partial \Omega$, denote by
    \[ 
    A(\zeta) = \int_{\Gamma_2 (\zeta)} \delta(z) |F'(z)|^2 \, dxdy \quad (z = x+iy),
    \]
    where $\Gamma_\alpha(\zeta)$ is the cone with vertex $\zeta$ given by 
    \[
   \Gamma_\alpha(\zeta)
     = \{\zeta + x+iy:  |x| < (1+\alpha) y, \ \ (x,y) \in \RR^2\}.
   \]
Let  $z\in \Gamma_0(\zeta)$.  We claim that
\begin{equation}\label{eq:cone-bound}
|F'(z)| \lesssim A(\zeta)^{1/2}\delta(z)^{-3/2}.
\end{equation}
To see this, consider $D$, the disk of radius $\delta(z)/100$ around $z$. Then
since $|F'|^2$ is subharmonic,  and $D\subset \Gamma_2(\zeta)$, we have
\[
|F'(z)|^2 \leq \frac1{|D|} \int_D |F'(\zeta)|^2 \, d\xi d\eta
\lesssim \delta(z)^{-3} \int_{D} \delta(\zeta)|F'(\zeta)|^2 \, d\xi d\eta 
\leq  \delta(z)^{-3} A(\zeta).
\]
Next, consider any $z$, $w\in \Gamma_0(\zeta)$
such that $\delta(z)/10 \le \delta(w) \le 10 \delta(z)$. Then, 
integrating along the line $z_t = tz + (1-t)w$, and applying \eqref{eq:cone-bound},
we have 
\[
|F(z) - F(w)| = \left|\int_0^1 (z-w)F'(z_t) \, dt \right|
\lesssim A(\zeta)^{1/2} \delta(z)^{-1/2}.
\]
Take  any  sequence $w_k\in \Gamma_0(\zeta)$ such that
$2^k \delta(z) \le \delta(w_k) \le 2^{k+1}\delta(z)$, $k = 1, \, 2, \dots$.  
Then  the preceding estimate yields 
\[
|F(w_{k+1}) - F(w_k)| \lesssim A(\zeta)^{1/2}2^{-k/2} \delta(z)^{-1/2},
\]
and hence the limit
\[
c(\zeta) = \lim_{k\to \infty} F(w_k)
\]
exists.  Since this is valid for any such sequence, we have 
\[
\lim_{|w|\to\infty} (F(w)  - c(\zeta)) = 0  \quad \mbox{for} \ \ w \in \Gamma_0(\zeta).
\]
Furthermore, the cones $\Gamma_0(\zeta)$ and $\Gamma_0(\zeta')$ overlap for sufficiently large $|w|$, so $c(\zeta)$ is a constant, independent of $\zeta\in \p \Omega$.

We also have for all $z\in \Gamma_0(\zeta)$, the estimate 
\begin{align*}
|F(z) - c| &  \leq  |F(z) - F(w_1)| +  \sum_{k=}^\infty |F(w_{k+1}- F(w_k)| \\
&  \lesssim \ 
A(\zeta)^{1/2} \delta(z)^{-1/2} \,
 \sum_{k=0}^\infty 2^{-k/2}  \ \lesssim \  A(\zeta)^{1/2} \delta(z)^{-1/2}.
\end{align*}
Lastly, note that the monotone convergence theorem, and the fact that  the full area integral,
\[
\int_{\Omega} \delta(\zeta) |F'(\zeta)|^2 \, \D \xi \, \D \eta = A,
\]
is finite imply $A(\zeta) \to 0$ as $|\zeta| \to \infty$, $\zeta \in \p \Omega$. 
In all, if $w\in \Omega$, $\delta(w) \ge\delta_0>0$, and $|w| \to \infty$, then either
$\delta(w)\to \infty$ in which case the estimate
\[
|F(w) - c| \lesssim A(\zeta)^{1//2} \delta(w)^{-1/2} \le A^{1/2} \delta(w)^{-1/2} \to 0,
\]
or else $w \in \Gamma_0(\zeta)$ with $|\zeta| \to \infty$, and we have 
\[
|F(w) - c| \lesssim A(\zeta)^{1//2} \delta(w)^{-1/2} \lesssim A(\zeta)^{1//2} \delta_0^{-1/2} \to 0.
\]
This concludes the proof of the proposition, and hence also Corollary \ref{cor:2.2written}.
\end{proof}

  \clearpage
  \nocite{*}
    \printbibliography[title={References}]
    
\end{document}